\documentclass[11pt,reqno]{amsart}
\usepackage{amssymb}
\usepackage{amsfonts}
\usepackage{mathrsfs}
\usepackage{amsmath}
\usepackage{graphicx}
\usepackage{hyperref}
\usepackage{float}
\usepackage{epstopdf}
\usepackage{color}
\usepackage{bm}
\usepackage{comment}
\usepackage{soul}
\usepackage{cite}

\allowdisplaybreaks
\newtheorem{theorem}{Theorem}[section]
\newtheorem{proposition}[theorem]{Proposition}
\newtheorem{lemma}[theorem]{Lemma}
\newtheorem{corollary}[theorem]{Corollary}

\newtheorem{example}[theorem]{Example}
\newtheorem{definition}[theorem]{Definition}

\newtheorem{thm}{Theorem}

\def\diam{\mathrm{diam}}

\def\mcD{\mathcal{D}}
\def\mcE{\mathcal{E}}
\def\mcF{\mathcal{F}}
\def\msG{\mathscr{G}}

\def\USC{\mathcal{USC}}
\def\mcB{\mathcal{B}}

\numberwithin{equation}{section}

\begin{document}
\title[Uniqueness and convergence of resistance forms on $\USC$'s]{Uniqueness and convergence of resistance forms on unconstrained Sierpinski carpets}

\author{Shiping Cao}
\address{Department of Mathematics, Cornell University, Ithaca 14853, USA}
\email{sc2873@cornell.edu}
\thanks{}

\author{Hua Qiu}
\address{Department of Mathematics, Nanjing University, Nanjing, 210093, P. R. China.}
\thanks{The research of Qiu was supported by the National Natural Science Foundation of China, grant 12071213, and the Natural Science Foundation of Jiangsu Province in China, grant BK20211142.}
\email{huaqiu@nju.edu.cn}

%    General info

\date{}


\maketitle

\begin{abstract}
We prove the uniqueness of self-similar $D_4$-symmetric resistance forms on unconstrained Sierpinski carpets ($\USC$'s). Moreover, on  a sequence of $\USC$'s $K_n, n\geq 1$ converging in Hausdorff metric, we show that the associated diffusion processes converge in distribution if and only if the geodesic metrics on $K_n, n\geq 1$ are equicontinuous with respect to the Euclidean metric.
\end{abstract}

\section{Introduction}
In a previous paper \cite{CQ3}, the authors constructed self-similar $D_4$-Dirichlet forms on a class of fractals named unconstrained Sierpinski carpets ($\USC$'s). Compared with the classical Sierpinski carpets in analysis on infinitely ramified fractals\cite{BB,BB1,BB4,BB2,BB3,BBKT,KZ}, we allow the level-$1$ cells of $\USC$'s to live off the $1/k$ grids, so that two cells may intersect at a line of irrational length. See Figure \ref{fig1} for some examples.

\begin{figure}[htp]
    \includegraphics[width=4.9cm]{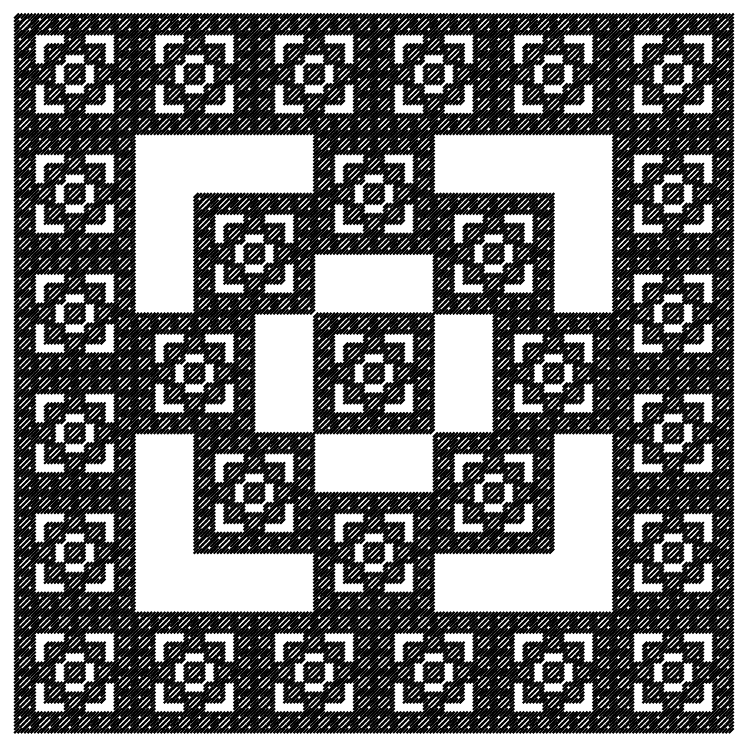}\hspace{1cm}
    \includegraphics[width=4.9cm]{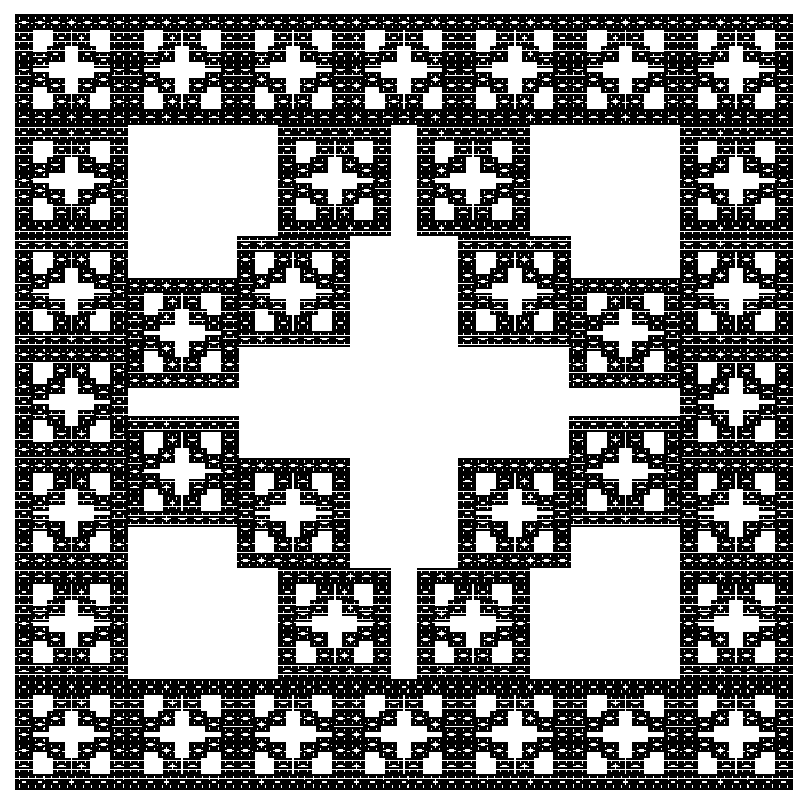}
    \caption{Unconstrained Sierpinski carpets ($\USC$'s).}
    \label{fig1}
\end{figure}

On $\USC$'s, the Knight move and corner move arguments developed by Barlow and Bass \cite{BB,BB3,BBKT} are no longer available, and the method of resistance estimate in \cite{BB4} does not work well. Hence, the uniqueness of self-similar $D_4$-symmetric resistance forms on $\USC$'s is not obvious. As a first goal of the paper, we will prove the following Theorem \ref{thm61}, where $L_2,L_4$ are the left and right border lines of a $\USC$.

\begin{thm}\label{thm61}
    Let $K$ be a $\USC$. There is a unique self-similar $D_4$-symmetric resistance form $(\mcE,\mcF)$ on $K$ such that the associated resistance metric $R$ is jointly continuous on $K\times K$ and satisfies $R(L_2,L_4)=1$.
\end{thm}

Second, as an application, by using the tool of $\Gamma$-convergence (see the book \cite{D} for basic knowledge, and \cite{C,KS,Mosco} for related papers), we will prove that the associated diffusion processes converge in distribution if the $\USC$'s converges in Hausdorff metric and under suitable geometric condition. In particular, we will prove Theorem \ref{thm2} below (see Theorem \ref{thm71} for a detailed statement), and in Subsection \ref{app2}, we will see that the convergence of resistance metrics implies the convergence of processes.

\begin{thm}\label{thm2}
    Let $K$, $K_n$, $n\geq 1$ be $\USC$'s such that $K_n$ converges to $K$ in Hausdorff metric, and $R$, $R_n$, $n\geq 1$ be the resistance metrics on them according to Theorem \ref{thm61}. Then
    $R_n\rightarrowtail R$
    if and only if the geodesic metrics $d_{G,n}$ on $K_n$, $n\geq 1$ are equicontinuous.
\end{thm}

We briefly introduce the structure of the paper.

First, as the preliminary parts, in Section \ref{sec2}, we introduce the precise definition and some geometric properties of $\USC$'s; in Section \ref{resis}, we review the definition of resistance forms and resistance metrics of Kigami \cite{ki3,ki4}, and recall some facts about the $\Gamma$-convergence of resistance forms.

Then, in Sections \ref{sec4} and \ref{sec5}, we establish the effective resistance estimates, where the constants depend only on $1/k$, the side length of level-$1$ cells, using the geodesic metrics. The proof of the lower bound estimate of effective resistances uses a trace theorem established in Section \ref{sec4}, which allows us to construct nice bump functions. The proof of the upper bound estimate of effective resistances uses the self-similar property of the form and the fact that the renormalization factor $r$ is smaller than $1-\frac{1}{k^2}$.

Finally, in Section \ref{sec6}, we prove the uniqueness theorem, using the idea of \cite{BBKT} and the effective resistance estimates. In Section \ref{sec7}, we prove the convergence of resistance metrics. By the uniform resistance estimates, we know that the resistance metrics are equicontinuous if and only if the geodesic metrics are equicontinuous, so there is a subsequential limit resistance form. The main ingredient of the proof focuses on proving that the limit form is self-similar, so we can apply the uniqueness theorem.

\section{The geometry of unconstrained Sierpinski carpets}\label{sec2}
In this section, we present a brief review of the geometric properties of $\USC$'s. For a point $x$ in $\mathbb{R}^2$, from time to time, we will write $x=(x_1,x_2)$ to specify its  two coordinates. For two points $x,y$ in $\mathbb{R}^2$, we will always write
\[\overline{x,y}=\big\{(1-t)x+ty:0\leq t\leq 1\big\}\]
 the line segment connecting them. The metric $d$, when we talk about points in $\mathbb{R}^2$, will always be the Euclidean metric on $\mathbb{R}^2$. In addition, we write
\[d(A,B)=\inf_{x\in A,y\in B}d(x,y),\]
for $A,B\subset \mathbb{R}^2$, which is always positive if $A,B$ are disjoint compact sets. We also write $d(x,A)=d(\{x\},A)$ for short. Readers should distinguish this with the \textit{Hausdorff metric} $\delta$:
\[
\delta(A,B)=\max\{\sup_{x\in B}d(x,A),\,\sup_{x\in A}d(x,B)\}.
\]

Let $\square$ be a unit square in $\mathbb{R}^2$. We let
\[q_1=(0,0),\quad q_2=(1,0),\quad q_3=(1,1),\quad q_4=(0,1)\]
be the four vertices of $\square$.
For convenience, we denote the group of self-isometries on $\square$ by
\[\msG=\big\{\Gamma_v,\Gamma_h,\Gamma_{d_1},\Gamma_{d_2},id,\Gamma_{r_1},\Gamma_{r_2},\Gamma_{r_3}\big\},\]
where $\Gamma_v,\Gamma_h,\Gamma_{d_1},\Gamma_{d_2}$ are \textit{reflections} ($v$ for vertical, $h$ for horizontal, $d_1,d_2$ for two diagonals),
\begin{equation}\label{eqn21}
\begin{aligned}
    &\Gamma_v(x_1,x_2)=(x_1,1-x_2),\qquad \Gamma_h(x_1,x_2)=(1-x_1,x_2),\\
    &\Gamma_{d_1}(x_1,x_2)=(x_2,x_1),\qquad\Gamma_{d_2}(x_1,x_2)=(1-x_2,1-x_1),
\end{aligned}
\end{equation}
 for $x=(x_1,x_2)\in \square$, $id$ is the identity mapping, and $\Gamma_{r_1},\Gamma_{r_2},\Gamma_{r_3}$ are \textit{rotations},
\begin{equation}\label{eqn22}
    \Gamma_{r_1}(x_1,x_2)=(1-x_2,x_1),\quad \Gamma_{r_2}=(\Gamma_{r_1})^2,\quad \Gamma_{r_3}=(\Gamma_{r_1})^3,
\end{equation}
around the center of $\square$ counterclockwisely with angle $\frac{j\pi}{2}$, $j=1,2,3$.
In this paper, we will focus on  structures with $\msG$-symmetry.

\begin{definition}[Unconstrained Sierpinski carpets]\label{def21}\quad

Let $k\geq 3$ and $4(k-1)\leq N\leq k^2-1$. Let $\{\Psi_i\}_{1\leq i\leq N}$ be a finite set of similarities in the form $\Psi_i(x)=\frac x k+ c_i$, $c_i\in\mathbb{R}^2$. Assume the following holds:

\noindent\emph{(Non-overlapping).} $\Psi_i(\square)\cap \Psi_j(\square)$ is either a line segment, or a point, or empty, $i\neq j$.

\noindent\emph{(Connectivity). } $\bigcup_{i=1}^N \Psi_i(\square)$ is connected.

\noindent\emph{(Symmetry).} $\Gamma\big(\bigcup_{i=1}^N \Psi_i(\square)\big)=\bigcup_{i=1}^N \Psi_i(\square)$ for any $\Gamma\in \msG$.

\noindent\emph{(Boundary included).} $\overline{q_1,q_2}\subset\bigcup_{i=1}^N \Psi_i(\square)\subset\square$.

Then, call the unique compact subset $K\subset \square$ such that
\[K=\bigcup_{i=1}^N \Psi_iK\]
an \emph{unconstrained Sierpinski carpet} ($\USC$).
\end{definition}

Here $k\geq 3$ is to avoid trivial set by the symmetry condition, $N\geq 4(k-1)$ is a requirement of the boundary included condition, and  $N\leq k^2-1$ ensures that  we are dealing with a non-trivial planar self-similar set. Note that when $k=3$, $N=8$, $K$ is the standard Sierpinski carpet. See Figure \ref{figure1} for more  examples of $\USC$.

\begin{figure}[htp]
    \includegraphics[width=4.9cm]{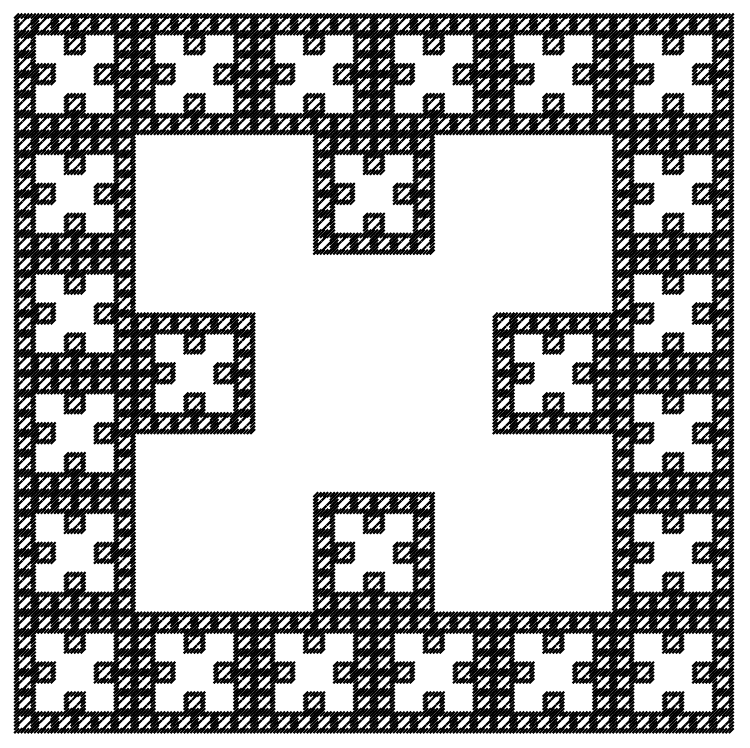}\hspace{0.2cm}
    \includegraphics[width=4.9cm]{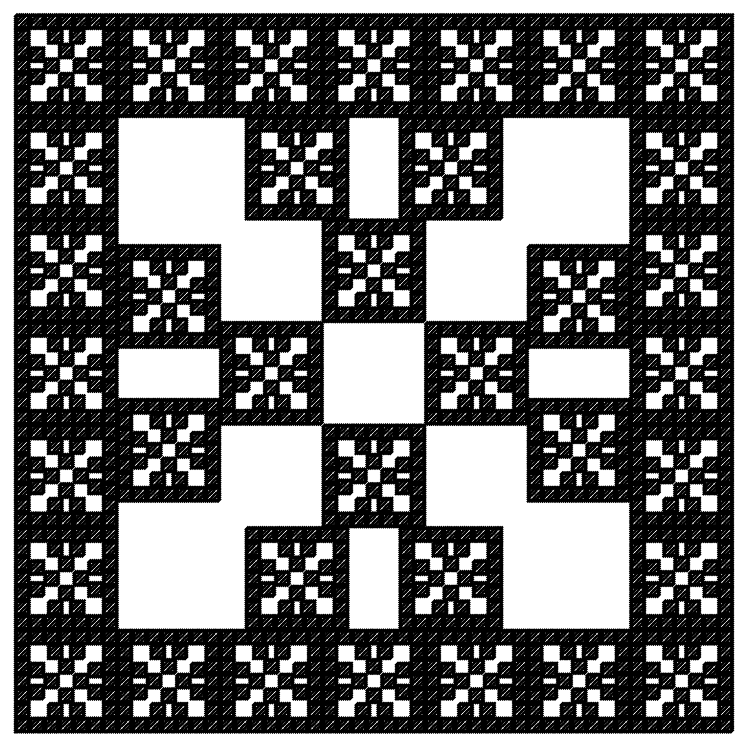}\hspace{0.2cm}
    \includegraphics[width=4.9cm]{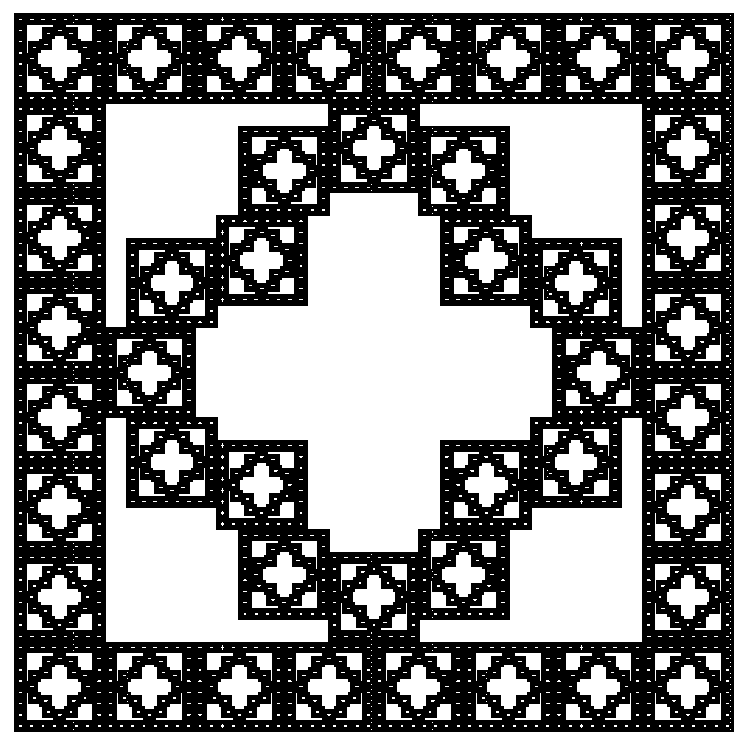}
    \caption{More $\USC$'s.}
    \label{figure1}
\end{figure}

 %Below we highlight some basic settings.\smallskip
\vspace{0.2cm}
\noindent\textbf{Basic settings and notations.}\vspace{0.2cm}

 (a). Throughout the paper, $k,N$ are fixed numbers.

(b). We number squares $\Psi_i(\square)$ along  the boundary of $\square$  in the following way. For $0\leq j\leq 3$ and $1\leq i\leq k-1$, we stipulate that
\begin{equation}\label{eqn23}
    \Psi_{(k-1)j+i}(x)=\frac{1}{k}(x-q_{j+1})+q_{j+1}+\frac{i-1}{k}(q_{j+2}-q_{j+1})
\end{equation}
with cyclic notation $q_5=q_1$.

(c). For convenience, we denote the four sides of $\square$ by
\begin{equation}
L_1=\overline{q_1,q_2},\quad L_2=\overline{q_2,q_3},\quad L_3=\overline{q_3,q_4},\quad L_4=\overline{q_4,q_1}.
\end{equation}

(d). We write $\partial_0 K=\bigcup_{i=1}^4 L_i$ for the square boundary of $K$.
\vspace{0.2cm}

Next, we introduce more notations about the word system.

\begin{definition}\label{def22}
    Let $W_0=\{\emptyset\},W_1=\{1,2,\cdots,N\}$ be the alphabet associated with $K$,  and for $n\geq 1$, $W_n=W_1^n:=\{w=w_1w_2\cdots w_n:w_i\in W_1\hbox{ for }1\leq i\leq n\}$ be the collection of \emph{words} of length $n$. Also, write $W_*=\bigcup_{n=0}^\infty W_n$ for the collection of finite words.

    (a). For every $w=w_1\cdots w_n\in W_n$, we denote $|w|=n$ the \emph{length} of $w$, and write
    \[\Psi_w=\Psi_{w_1}\circ \cdots \circ \Psi_{w_n}.\]
    Each $w\in W_n$ represents an \emph{$n$-cell} $\Psi_wK$ in $K$.

    (b). For $n\geq 0$, we define
    \[\partial W_n=\big\{w\in W_n: \Psi_wK\cap \partial_0 K\neq \emptyset\big\}.\]
    Also, we write
    \[W_{n,i}=\big\{w\in W_n:\Psi_wK\cap L_i\neq\emptyset\big\},\text{ for } i=1,2,3,4,\]
    and $W_{*,i}=\bigcup_{n=0}^\infty W_{n,i}$.
\end{definition}

In the previous paper \cite{CQ3}, the authors have shown that a $\USC$ $K$ has the following geometric properties (A1)-(A4). In particular, (A1), (A2), (A4) are the same  in \cite[Section 2]{KZ}, while (A3) is exactly the condition that ``\textit{the Euclidean metric is $2$-adapted}'' introduced in \cite{ki5} by Kigami.

\vspace{0.2cm}

(A1). \emph{The open set condition: there exists a non-empty open set $U\subset \mathbb{R}^2$ such that $\bigcup_{i=1}^N \Psi_iU\subset U$ and $\Psi_iU\cap \Psi_jU=\emptyset$ for every $i\neq j\in \{1,\cdots,N\}$. }

(A2). \emph{For each $w,w'\in W_n$, there is a finite collection of words $\{w=w^{(1)},\cdots, w^{(l)}=w'\}$ in $W_n$ satisfying $\Psi_{w^{(i)}}K\cap \Psi_{w^{(i+1)}}K\neq \emptyset$ for $1\leq i<l$.}

(A3). \emph{There is a constant $0<c_0<1$ satisfying the following.}

\emph{(a). If $x,y\in K$ and $d(x,y)< c_0k^{-n}$, then there exist $w, w',w''\in W_n$ such that $x\in \Psi_{w}K$, $y\in \Psi_{w'}K$ and $\Psi_{w}K\cap \Psi_{w''}K\neq \emptyset$, $\Psi_{w'}K\cap \Psi_{w''}K\neq \emptyset$. }

\emph{(b). If $w,w'\in W_n$ and there is no $w''\in W_n$ so that
    \[\Psi_wK\cap \Psi_{w''}K\neq \emptyset, \quad \Psi_{w'}K\cap \Psi_{w''}K\neq \emptyset,\]
then $d(x,y)\geq c_0 k^{-n}$ for any $x\in \Psi_wK$ and $y\in \Psi_{w'}K$. }

(A4). \emph{$\partial W_n\neq W_n$ for $n\geq 2$.}\vspace{0.2cm}

Finally, at the end of this section, we prove that $K\setminus\partial_0K$ is a uniform domain. Here, we follow the definition of \cite{Mathav}.

\begin{definition}\label{def23}
Let $(X,d)$ be a metric space. We call $U\subset X$ a \emph{uniform domain}, if $U$ is a connected, open, proper subset of $X$, and for some $\alpha\in (1,\infty)$, for every $x,y\in U$, there exists a curve $\gamma\subset U$ from $x$ to $y$ such that
\begin{align*}
    {\rm diam}(\gamma)&\leq \alpha\,d(x,y),\\
    d(z,X\setminus U)&\geq \alpha^{-1}\min\{d(z,x),d(z,y)\}\quad\hbox{ for every }z\in \gamma.
\end{align*}
\end{definition}

\begin{proposition}\label{prop24}
Let $K$ be a $\USC$, then $K\setminus \partial_0K$ is a uniform domain.
\end{proposition}
\begin{proof}
For $x\neq y\in K\setminus\partial_0K$, we need to find a curve $\gamma$ that satisfies the requirements in Definition \ref{def23}.  Let
\[
\square_n:=[k^{-n},1-k^{-n}]^2\hbox{ and }\partial\square_n:=[k^{-n},1-k^{-n}]^2\setminus (k^{-n},1-k^{-n})^2\quad\text{ for }n\geq 1.
\]
We consider two different cases based on the locations of $x$ and $y$.\vspace{0.2cm}

\textit{{Case 1}. $d(x,y)\in [c_0k^{-n-1},c_0k^{-n})$ for some $n\geq 2$ and $x\in \square_{n-1},y\in\square_{n-1}$.}\vspace{0.2cm}

In this case, by (A3), we can find $w,w',w''\in W_n$ such that $x\in \Psi_wK$, $y\in \Psi_{w'}K$, $\Psi_wK\cap \Psi_{w''}K\neq \emptyset$ and $\Psi_{w'}K\cap \Psi_{w''}K\neq \emptyset$. By the assumption that $x\in \square_{n-1},y\in\square_{n-1}$, we can see that $d(z,\partial_0K)\geq k^{n}$ for each $z\in \Psi_wK\cup \Psi_{w'}K\cup \Psi_{w''}K$.  It suffices to choose a curve $\gamma\subset \Psi_wK\cup \Psi_{w'}K\cup \Psi_{w''}K$ connecting $x,y$.\vspace{0.2cm}

\textit{{Case 2.} $d(x,y)\geq c_0k^{-n-1}$ and $\{x,y\}\setminus\square_{n-1}\neq\emptyset$ for some $n\geq 2$.} \vspace{0.2cm}

In this case, we let $m'\in\mathbb{Z}$ such that $d(x,y)\in [c_0k^{-m'-2},c_0k^{-m'-1})$, and let $m=\max\{m',1\}$. Clearly, $m\leq n-1$. Without loss of generality, we assume that $x\in K\setminus\square_{n-1}\subset K\setminus \square_m$. We claim that there is $x'\in\partial\square_m$ and a curve $\gamma_1$ connecting $x,x'$ so that
\begin{align*}
    {\rm diam}(\gamma_1)&\leq C_1\,k^{-m},\\
    d(z,\partial_0K)&\geq C_2\,d(z,x)\quad\hbox{ for every }z\in \gamma_1,
\end{align*}
where $C_1,C_2$ are constants that depend only on $k$. In fact, let $l\geq 0$ such that $x\in \square_{m+l+1}\setminus \square_{m+l}$, set $x=x_{l+1}$, and for $i=0,1,\cdots,l$, we can find a sequence $x_i\in \partial\square_{m+i}$ and curve $\gamma_{1,i}$ connecting $x_i,x_{i+1}$ such that
\begin{align*}
    {\rm diam}(\gamma_{1,i})&\leq 2^{1/2}k^{-m-i},\\
    d(z,\partial_0K)&\geq k^{-m-i-1}\quad\hbox{ for every }z\in \gamma_{1,i}.
\end{align*}
It suffices to take $\gamma_1=\gamma_{1,0}\cup\gamma_{1,1}\cup\cdots\cup\gamma_{1,i}$ and $x'=x_0$.

By a same argument, if $y\in K\setminus\square_m$, we can find $y'\in\partial\square_m$ and a curve $\gamma_3$ connecting $y,y'$ so that
\begin{align*}
    {\rm diam}(\gamma_3)&\leq C_1\,k^{-m},\\
    d(z,\partial_0K)&\geq C_2\,d(z,y)\quad\hbox{ for every }z\in \gamma_3;
\end{align*}
otherwise, if $y\in \square_m$, we can find $y'\in \partial\square_m$ such that $d(y,y')\leq d(x,y)$, and we can use the argument of Case 1 to find a curve $\gamma_3$ connecting $y,y'$ so that $\diam(\gamma_3)<C_3d(y,y')$ and $d(z,\partial_0K)\geq k^{-m-1}\geq C_4d(x,y)$ for each $z\in \gamma_3$, where $C_3,C_4$ depend only on $K$.

Finally, we choose $\gamma_2$ to be the shortest curve on $\partial\square_m$ that connects $x',y'$, and take $\gamma=\gamma_1\cup\gamma_2\cup\gamma_3$.
\end{proof}

\section{Review of resistance forms}\label{resis}
In this section, we review the concepts of resistance forms and effective resistances. Readers can read \cite{ki3,ki4} for a systematic study. We begin with the basic definition of resistance forms. \vspace{0.2cm}

\noindent\textbf{Resistance forms.} Let $X$ be a set, and $l(X)$ be the space of all real-valued functions on $X$. A pair $(\mathcal{E},F)$ is called a \textit{non-degenerate resistance form} on $X$ if it satisfies the following  conditions:

\noindent (RF1). \textit{$\mathcal{F}$ is a linear subspace of $l(X)$ containing constants and $\mathcal{E}$ is a nonnegative definite bilinear  form on $\mathcal F$; $\mathcal{E}(f):=\mathcal{E}(f,f)=0$ if and only if $f$ is constant on X.}

\noindent (RF2). \textit{Let $\sim$ be an equivalent relation on $\mathcal{F}$ defined by $f\sim g$ if and only if $f-g$ is constant on X. Then $(\mathcal{F}/\sim, \mathcal{E})$ is a Hilbert space.}

\noindent (RF3). \textit{For any finite subset $V\subset X$ and for any $u\in l(V)$, there exists $f\in \mathcal{F}$ such that $f|_V=u$.}

\noindent (RF4). \textit{For any $p,q\in X$, $\sup\{\frac{|f(p)-f(q)|^2}{\mathcal{E}(f)}: f\in \mathcal{F},\mathcal{E}(f)>0\}$ is finite.}

\noindent (RF5). (Markov property) \textit{If $f\in \mathcal{F}$, then $\bar{f}={ \min\{\max\{f,0\}, 1\}}\in \mathcal{F}$ and $\mathcal{E}(\bar{f})\leq \mathcal{E}(f)$.}\vspace{0.2cm}

\noindent\textbf{Effective resistances.}  Let $(\mcE,\mcF)$ be a resistance form on a set $X$. We introduce the \textit{effective resistance} between two disjoint subsets $A,B$ of $X$, defined as
\[R(A,B):=\sup\{\big(\mcE(f)\big)^{-1}:\,f\in \mcF,\,f|_A=0,\,f|_B=1\}.\]
For a pair of points $x,y$ in $K$, we write
\[
R(x,y)=R(\{x\},\{y\})
\]
for short if $x\neq y$, and $R(x,y)=0$ if $x=y$. It is proved that $R(x,y),x,y\in X$ is a metric on $X$ \cite[Theorem 2.1.14]{ki3}, and we call $R$ the \emph{resistance metric} associated with $(\mcE,\mcF)$.

More generally, there is a way of defining resistance metrics directly. See \cite[Definition 2.3.2]{ki3} for details. We denote the collection of resistance metrics on $X$ by $\mathcal{RM}(X)$ in the remaining  of this section.

It is known (see  \cite[Theorem 2.3.7]{ki3} by Kigami) that if $R\in \mathcal{RM}(X)$ so that $(X,R)$ is a separable metric space, there is a unique \textit{resistance form} $(\mcE,\mcF)$ on $X$ such that $\mcF\subset C(X,R)$ (the space of continuous functions on $X$ with respect to $R$) and
\[R(x,y)^{-1}=\inf\big\{\mcE(f):f(x)=0,f(y)=1, f\in\mcF\big\} \quad \hbox{ for every } x,y\in X.\]
The form is called \textit{regular} if $\mcF$ is dense in $C(X,R)$. If $(X,R)$ is compact, then $(\mcE,\mcF)$ is always regular (see \cite[ Corollary 6.4]{ki4}). In addition, if $\mu$ is a Radon measure on $X$ with full support, then $(\mcE,\mcF)$ becomes a regular Dirichlet form on $L^2(X,\mu)$ (see  \cite[Theorem 9.4]{ki4}).\vspace{0.2cm}

\noindent\textbf{Remark 1.} We always additionally assume that $X$ has an original metric $d$, and the resistance metric $R$ is jointly continuous on $(X,d)$, i.e.
\[
R(x_n,y_n)\to R(x,y) \text{ for every } x_n\to x, y_n\to y \text{ in }d, \text{ as } n\to \infty.
\]
With this assumption, we have $\mcF\subset C(X)$, where $C(X)$ is the space of continuous functions on $X$ in the metric $d$.
In the case that $(X,d)$ is compact, we additionally have $\mcF$ is a dense subset of $C(X)$ by \cite[Corollary 6.4]{ki4}. As a consequence, $R(A,B)\in (0,\infty)$ for any pair of disjoint closed subsets $A,B$ of $X$: $R(A,B)>0$ by the assumption that $\mcF$ is dense in $C(X)$, and $R(A,B)<\infty$ noticing that $R(A,B)\leq R(x,y)<\infty$ for $x\in A,y\in B$.\vspace{0.2cm}

\noindent\textbf{Remark 2.} We may also regard a resistance form $(\mcE,\mcF)$ as  a \textit{quadratic form} $\mcE$ on $C(X)$ by extending to define
\[\mcE(f)=
\begin{cases}
    \mcE(f,f),&\text{ if }f\in \mcF,\\
    \infty,&\text{ if }f\in C(X)\setminus\mcF.
\end{cases}\]

In this paper, we will focus on self-similar $D_4$-symmetric resistance forms on a $\USC$ $K$. See the definition below.

\begin{definition}\label{def3.1}
Let $K$ be a $\USC$ and $(\mcE,\mcF)$ be a resistance form on $K$.

(a). We say that $(\mcE,\mcF)$ is \emph{self-similar}, if
\[
\mathcal{F}= \{f\in l(K): f\circ \Psi_i\in \mathcal{F}\hbox{ for every } 1 \leq i \leq N\},
\]
and the energy self-similar identity  holds,
\begin{equation}
    \mcE(f)=\sum_{i=1}^N r^{-1}\mcE(f\circ \Psi_i)\quad\hbox{ for every } f\in \mcF,
\end{equation}
where $r\in (0,\infty)$ is called a \emph{renormalization factor}.

(b). We say that $(\mcE,\mcF)$ is $D_4$-\emph{symmetric} if for every $f\in \mcF$ and $\Gamma\in \mathscr{G}$, we have $f\circ\Gamma\in\mcF$ and $\mcE(f\circ\Gamma)=\mcE(f)$, where $\mathscr{G}$ is the group of self-isometrics on $\square$ as defined in Section \ref{sec2}.
\end{definition}

\subsection{Convergence of resistance forms}\label{app1}
We recall some facts about convergence of resistance forms on varying spaces based on one of the authors' previous work \cite{C}. Readers are also invited to read \cite{C1} by Croydon and \cite{CHK} by Croydon, Hambly and Kumagai to see a proof on the weak convergence of associated diffusion processes. Also, see \cite{KS} for a general study of Mosco convergence on varying $L^2$-spaces.

\vspace{0.2cm}
\noindent\textbf{Basic settings.}\vspace{0.2cm}

(a). Let $(B,d)$ be a compact metric space with a compact subset $A$, and let $\{A_n\}_{n\geq 1}$ be a sequence of compact subsets in $B$ such that
    \[\lim_{n\to\infty}\delta(A_n, A)=0,\]
    where $\delta$ is the \textit{Hausdorff metric} on compact subsets in $B$. Write $A_n\to A$ for short.

(b). For any topological space $X$, $\|\cdot\|_{C(X)}$ denotes the supremum norm on $C(X)$.

(c). Let $f_n\in C(A_n)$ for $n\geq 1$. \\
We say $\{f_n\}_{n\geq 1}$ is \textit{uniformly bounded} if $\sup_{n\geq1}\|f_n\|_{C(A_n)}<\infty$. \\
We say $\{f_n\}_{n\geq 1}$ is \textit{equicontinuous} if
$\lim\limits_{\eta\to 0}\sup_{n\geq 1}\big\{\big|f_n(x)-f_n(y)\big|:x,y\in A_n,d(x,y)<\eta\big\}=0$.

(d). Let $f_n\in C(A_n)$ for $n\geq 1$ and $f\in l(A)$. We write $f_n\rightarrowtail f$ if $f(x)=\lim\limits_{n\to\infty}f_n(x_n)$ for any $x\in A$ and $x_n\in A_n,n\geq 1$ such that $x_n\to x$ as $n\to\infty$. \vspace{0.2cm}

\noindent\textbf{Remark.} We also need to consider the Cartesian product $B^2=\{(x,y):x,y\in B\}$, equipped with the metric $d_{B^2}\big((x_1,y_1),(x_2,y_2)\big)=d(x_1,x_2)+d(y_1,y_2)$ (or the equivalent metric $\sqrt{|d(x_1,x_2)|^2+|d(y_1,y_2)|^2}$). Under the basic settings, we immediately have $A_n^2\to A^2$ in the Hausdorff metric on compact subsets in $B^2$.

Resistance metrics are an important class of continuous functions on $A^2$ (or $A_n^2$). In particular, for each metric $R$ on $A$, by the triangle inequality, $\big|R(x_1,y_1)-R(x_2,y_2)\big|\leq R(x_1,x_2)+R(y_1,y_2)$, so it suffices to study $\sup_{x,y\in A:d(x,y)<\eta}R(x,y)$ to understand the continuity modulus of $R$ on $A^2$.\vspace{0.2cm}

In  \cite[Section 2.1]{C}, it is show that $f_n\rightarrowtail f$ is a natural extension of the concept of uniform convergence, and there is an analogue of Arzel\`a–Ascoli theorem.

\begin{lemma}\cite[Lemma 2.2, Proposition 2.3]{C}\label{lemmab2} Let $f_n\in C(A_n)$ for $n\geq 1$.

    (a).  If $\{f_n\}_{n\geq 1}$ is uniformly bounded and equicontinuous, then there is a subsequence $\{f_{n_l}\}_{l\geq 1}$ and $f\in C(A)$ such that $f_{n_l}\rightarrowtail f$.

    (b). If $f_{n}\rightarrowtail f$ for some $f\in l(A)$, then $\{f_n\}_{n\geq 1}$ is uniformly bounded and equicontinuous, and $f\in C(A)$.

    (c). Let $f\in C(A)$. Let $\bar{A}=(\{0\}\times A)\bigcup\big(\bigcup_{n\geq 1}(\{\frac{1}{n}\}\times A_n)\big)$ with topology induced from $[0,1]\times B$. Define $\bar{f}\in l(\bar{A})$ by
    \[\bar{f}(t,x)=\begin{cases}
        f_n(x),&\text{ if }t=\frac{1}{n},\\
        f(x),&\text{ if }t=0.
    \end{cases}\]The following claims are equivalent:

    \hspace{0.2cm}(c-i). $f_n\rightarrowtail f$;

    \hspace{0.2cm}(c-ii). $\bar{f}\in C(\bar{A})$;

    \hspace{0.2cm}(c-iii). there exists $\{g_n\}_{n\geq 1}\cup \{g\}\subset C(B)$ such that $g_n|_{A_n}=f_n$ for every $n\geq 1$, $g|_A=f$ and $g_n\rightrightarrows g$. Here $\rightrightarrows$ denotes uniform convergence.
\end{lemma}

An important consequence of the lemma is that we can create resistance metrics at the limit.
\begin{proposition}\cite[Theorem 2.9]{C}\label{thmb3}
    Let $R_n\in \mathcal{RM}(A_n)$ for $n\geq 1$, and assume
    \[\psi_1\big(d(x,y)\big)\leq R_n(x,y)\leq \psi_2\big(d(x,y)\big),\quad\text{ for all } n\geq 1, x,y\in A_n, \]
    for some $\psi_1,\psi_2\in C[0,\infty)$ with $\psi_1(0)=\psi_2(0)=0$ and $\psi_2(t)\geq \psi_1(t)>0$ for $t>0$.
    Then, there exists $R\in \mathcal{RM}(A)$ and a subsequence $\{n_l\}_{l\geq 1}$ such that $R_{n_l}\rightarrowtail R$. In particular,
    \[\psi_1\big(d(x,y)\big)\leq R(x,y)\leq \psi_2\big(d(x,y)\big),\quad\text{ for all } x,y\in A.\]
\end{proposition}

\begin{definition}\label{defb4}
    Let $\mcE_n$ be quadratic forms  on $C(A_n)$, and $\mcE$ be a quadratic form  on $C(A)$. We say $\mcE_n$ \emph{$\Gamma$-converges} to $\mcE$ on $C(B)$ if and only if (a),(b) hold:

    (a). If $f_n\rightrightarrows f$ with $\{f_n\}_{n\geq 1}\cup\{f\}\subset C(B)$, then
    \[\mcE(f|_A)\leq \liminf_{n\to\infty} \mcE_n(f_n|_{A_n}).\]

    (b). For each $f\in C(B)$, there exists a sequence $\{f_n\}_{n\geq 1}\subset C(B)$ such that $f_n\rightrightarrows f$ and
    \[\mcE(f|_A)=\lim_{n\to\infty} \mcE_n(f_n|_{A_n}). \]
\end{definition}

\noindent\textbf{Remark 1.} In the following, we will also say a sequence of resistance forms $(\mcE_n,\mcF_n)$ on $A_n$ \textit{$\Gamma$-converges} to a resistance form $(\mcE,\mcF)$ on $A$, which exactly means the $\Gamma$-convergence on $C(B)$ for their associated quadratic forms with extended real values (see Remark 2 before Definition \ref{def3.1}).\vspace{0.2cm}

\noindent\textbf{Remark 2.}  $f_n\rightrightarrows f$ can be replaced with $f_n|_{A_n}\rightarrowtail f|_{A}$ by Lemma \ref{lemmab2} (c).\vspace{0.2cm}

A main result of \cite{C} is the following proposition stating that the convergence of resistance metrics will result in the $\Gamma$-convergence of the resistance forms.

\begin{proposition}\cite[Theorem 2.13]{C}\label{thmb5}
    Let $R_n\in\mathcal{RM}(A_n)\cap C(A_n^2)$ for $n\geq 1$, $R\in \mathcal{RM}(A_n)\cap C(A^2)$, and assume $R_n\rightarrowtail R$. Let $(\mcE_n,\mcF_n)$ be the resistance form associated with $R_n$ for $n\geq 1$, and let $(\mcE,\mcF)$ be the resistance form associated with $R$. Then, we have $(\mcE_n,\mcF_n)$ $\Gamma$-converges to $(\mcE,\mcF)$ on $C(B)$.
\end{proposition}

Next, we turn to the convergence of resolvent kernels. For convenience, we add the following requirements (e) and (f) into the basic settings (a)-(d) in Subsection \ref{app1}. \vspace{0.2cm}

(e). There is a sequence of resistance metrics $R_n\in  \mathcal{RM}(A_n)\cap C(A_n^2), n\geq 1$ and $R\in  \mathcal{RM}(A)\cap C(A^2)$. Let $(\mcE_n,\mcF_n)$ be the resistance form associated with $R_n$ for $n\geq 1$; let $(\mcE,\mcF)$ be the resistance form associated with $R$. Assume $R_n\rightarrowtail R$.

(f). Let $\mu_n$ be a sequence of Radon measures supported on $A_n$, $n\geq 1$ and $\mu$ be a Radon measure supported on $A$. Assume $\mu_n\Rightarrow\mu$, where ``$\Rightarrow$'' refers to weak convergence.\vspace{0.2cm}

    Let $(A,R)$ be a compact metric space with $R\in \mathcal{RM}(A)$, and let $\mu$ be a Radon measure on $(A,R)$. Assume $(A,R)$ is separable. Then for any $\alpha>0$ and any $x\in A$, there exists a unique function $u_{\alpha}(x,\cdot)\in \mcF$ such that
    \[\mcE_{\alpha}\big(u_\alpha(x,\cdot),f\big)=f(x),\quad\text{ for every } f\in \mcF,\]
    where $(\mcE,\mcF)$ is the resistance form associated with $R$, and $\mcE_{\alpha}(f,g)=\mcE(f,g)+\alpha\int_{A}fgd\mu$ for any $f,g\in \mcF$. The function $u_\alpha$ is called the \textit{resolvent kernel} associated with $R$ and $\mu$.
It is well-known that $u_\alpha$ is a symmetric function, i.e. $u_\alpha(x,y)=u_\alpha(y,x)$.

More importantly, since $(A,R)$ is compact, $(\mcE,\mcF)$ is a regular Dirichlet form on $L^2(A,\mu)$, so by \cite[Theorem 7.2.1]{FOT} there is a Hunt process $\bm{M}=(\Omega,\mathcal{M},X_t,\mathbb{P}_x)$ associated with $(\mcE,\mcF)$ on $A$, and we have
\begin{equation}\label{eqnc4}
    \int_E u_\alpha(x,y)\mu(dy)=\int_{0}^\infty e^{-\alpha t}\mathbb{P}_x(X_t\in E)dt,\quad\text{ for every } E\subset A.
\end{equation}
In other words, $u_\alpha(x,y)$ is the resolvent kernel associated with transition kernel $p_t(x,y)$.

\begin{proposition}\cite[Theorem 2.17]{C}\label{propc6}
    Assume all the basic settings. Let $\alpha>0$, then
    \[u_{\alpha,n}\rightarrowtail u_\alpha,\]
    where $u_{\alpha,n}$ is the resolvent kernel associated with $R_n,\mu_n$ for each $n\geq 1$, and $u_\alpha$ is the resolvent kernel associated with $R,\mu$.
\end{proposition}

\subsection{Weak convergence of  processes}\label{app2}
Following Subsection \ref{app1}, we still consider convergence of resistance metrics, but focus on the probability side.

We assume (a)-(f) of Subsection \ref{app1} in this subsection. Let $\bm{M}=(\Omega,\mathcal{M},X_t,\mathbb{P}_x)$ be the Hunt process associated with $(\mcE,\mcF)$ on $A$. We have \vspace{0.2cm}

i). \textit{$\bm{M}$ is unique} by \cite[Theorem 4.2.8]{FOT}, noticing that each point has positive capacity in our setting;

ii). \textit{$\bm{M}$ is Feller}: let $p_t(\cdot,\cdot)$ be the transition kernel associated with the process, i.e. $p_t(x,B)=\mathbb{P}_x(X_t\in B)$. Then $p_tf\in C(A)$ for any $f\in C(A)$ and $p_tf\rightrightarrows f$ as $t\to 0$, where $p_tf(x)=\int_A p_t(x,dy)f(y)$.
This can be verified easily for $f\in \mcF$, then extends to every $f\in C(A)$ noticing that $\mcF$ is dense in $C(A)$. \vspace{0.2cm}

By a same reason, for $n\geq 1$, we have a unique Feller process $\bm{M}_n=(\Omega_n,\mathcal{M}_n,X_t^{(n)},\mathbb{P}^{(n)}_x)$ associated with $(\mcE_n,\mcF_n)$. We denote the associated transition kernel by $p_t^{(n)}(\cdot,\cdot)$, and write $p^{(n)}_tf(x)=\int_{A_n} p^{(n)}_t(x,dy)f(y)$. The main result in this subsection is the following theorem concerning weak convergence.

\begin{theorem}\label{thmc1}
    Assume all the basic settings. Let $x_n\in A_n$ and $x\in A$. If $x_n\to x$ in $(B,d)$, then \[\mathbb{P}_{x_n}^{(n)}\big((X^{(n)}_t)_{t\geq 0}\in \cdot\big)\Rightarrow \mathbb{P}_x\big((X_t)_{t\geq 0}\in \cdot\big),\]
    where the weak convergence is in the sense of probability measures on $D(\mathbb{R}_+,B)$ (that is, the space of cadlag processes on $(B,d)$, equipped with the usual Skorohod $J1$-topology).
\end{theorem}

\noindent\textbf{Remark.} Theorem \ref{thmc1} is almost same as  \cite[Theorem 1.2]{C1}, with additional assumption about compactness. In \cite{C1,CHK}, it is assumed that $(A_n,R_n)$ and $(A,R)$ are embedded isometrically in some space $M$ by Gromov-Hausdorff vague convergence. However, our basic setting only tells us the Gromov-Hausdorff convergence of $(A_n,d)$, while the Gromov-Hausdorff convergence of $(A_n,R_n)$ to $(A,R)$ is not easy to prove (see  \cite[Section 2.5]{C} for a proof). Due to this,  we can not use \cite[Theorem 1.2]{C1} directly, but an essentially same idea can still leads to Theorem \ref{thmc1}. For the convenience of readers, we sketch the proof, which relies on two ingredients: \textit{tightness} and \textit{convergence of transition kernels} (finite dimensional distributions). We refer to  \cite[Lemma 4.3]{C1} for a neat result on the tightness.

\begin{lemma}\cite[Lemma 4.3]{C1}\label{lemmac2}
    Let $(A,R)$ be a compact metric space, where $R\in \mathcal{RM}(A)$; let $\mu$ be a Radon measure supported on $A$. Then, for any $\varepsilon>0$ and $0<\eta\leq \varepsilon/8$,
    \[
    \sup_{x\in A}\mathbb{P}_x\big(\sup_{s\leq t}R(x,X_s)<\varepsilon\big)\leq \frac{32N_R(A,\varepsilon/4)}{\varepsilon}\cdot\big(\eta+\frac{t}{\inf_{x\in A}\mu\big(B_R(x,\eta)\big)}\big),
    \]
    where $N_R(A,\varepsilon)$ is the minimal size of an $\varepsilon$-net of $(A,R)$, and $B_R(x,\eta)$ is a ball centered at $x$ with radius $\eta$ with respect to $R$.
\end{lemma}

The tightness of the processes follows immediately from Lemma \ref{lemmac2}.

Before proceeding, we point out an easy consequence of Lemma \ref{lemmab2} (c): ``(c-i)$\Rightarrow$ (c-iii)''. Assuming the basic settings, if $f_n\rightarrowtail f$, where $f_n\in C(A_n),n\geq 1$ and $f\in C(A)$, then
\begin{equation}\label{eqnc1}
    \int_{A_n} f_n(x)\mu_n(dx)\to \int_{A} f(x)\mu(dx).
\end{equation}
In fact, we have $\{g_n\}_{n\geq1}\cup\{g\}\subset C(B)$ such that $g_n|_{A_n}=f_n$ for $n\geq 1$, $g|_A=f$ and $g_n\rightrightarrows g$. Then $\int_{A_n} g_n(x)\mu_n(dx)=\int_{A_n} g(x)\mu_n(dx)+\int_{A_n} (g_n-g)(x)\mu_n(dx)\to \int_A g(x)\mu(dx)$, noticing $\lim_{n\to\infty}\mu_n(B)=\mu(B)<\infty$.

In particular, an easy application of (\ref{eqnc1}) leads to $f_n\mu_n\Rightarrow f\mu$.

\begin{proposition}\label{propc3}
    Assume all the basic settings. For any sequence
    $x_n,n\geq 1$ with $x_n\in A_n$, the laws of $X^{(n)}$ under $\mathbb{P}^{(n)}_{x_n},n\geq 1$ are tight in $D(\mathbb{R}_+,B)$.
\end{proposition}
\begin{proof}
    First, we let $\phi(s)=\sup\{R_n(x,y):n\geq 1,x,y\in A_n\text{ and }d(x,y)\leq s\}$ for $s\geq 0$. Then, $\phi(s)$ is continuous at $0$ by equicontinuity of $R_n$ (by Lemma \ref{lemmab2}), and $R_n(x,y)\leq\phi\big(d(x,y)\big)$ for any $n\geq 1$ and $x,y\in A_n$. Now, fix any $\varepsilon>0$, we can find $s>0$ such that $\phi(2s)\leq \varepsilon$, so
    \begin{equation}\label{eqnc2}
        N_{R_n}(A_n,\varepsilon)\leq N_d(B,s)<\infty,\quad\text{ for all  } n\geq 1,
    \end{equation}
    where $N_d(B,s)$ is the minimal size of an $s$-net of $(B,d)$. In fact, given an $s$-net $\{x_1,x_2,\cdots,x_m\}$ of $(A,d)$, we can simply pick a point in $B_d(x_j,s)\cap A_n$ for each $j$ (if the intersection is non-empty) to form a $2s$-net of $(A_n,d)$, hence an $\varepsilon$-net of $(A_n, R_n)$.

    Next, for $0<\eta\leq\varepsilon/8$, $n\geq 1$ define $g_n(x,y)=\big(1-\eta^{-1}R_n(x,y)\big)\vee 0$; define $g(x,y)=\big(1-\eta^{-1}R(x,y)\big)\vee 0$. Then since $R_n\rightarrowtail R$, we have $g_n\rightarrowtail g$. For $n\geq 1$, let $f_n(x)=\int_{A_n} g_n(x,y)\mu_n(dy)$; let $f(x)=\int_A g(x,y)\mu(dy)$. Then, $f_n\in C(A_n),n\geq 1$, $f\in C(A)$ and $f_n\rightarrowtail f$ by (\ref{eqnc1}). In particular, $f_n,f$ are all strictly positive, so by Lemma \ref{lemmab2} (c):(c-i)$\Rightarrow$(c-ii), we have
    \begin{equation}\label{eqnc3}
        \inf_{n\geq 1}\inf_{x\in A_n}\mu_n\big(B_R(x,\eta)\big)\geq \inf_{n\geq 1}\inf_{x\in A_n}\int g_n(x,y)\mu_n(dy)=\inf_{(t,x)\in \bar{A}}\bar{f}(t,x)>0,
    \end{equation}
    where $\bar{A}$ and $\bar{f}$ are defined with $A_n$, $f_n$ and $f$ as in Lemma \ref{lemmab2} (c). Hence, by Lemma \ref{lemmac2} and (\ref{eqnc2}), (\ref{eqnc3}), we immediately have
    \[\lim_{t\to0}\sup_{n\geq 1}\sup_{x\in A_n}\mathbb{P}_x^{(n)}(\sup_{s\leq t}d(x,X_s^{(n)})<\varepsilon)>0.\]
    This indicates the Aldous' tightness criterion (for example, one can use  \cite[Theorem 16.10]{Bi}, combined with \cite[Theorem 9.1, Chapter 3]{EK}) by the strong Markov property.
\end{proof}

Theorem \ref{thmc1} follows from Propositions  \ref{propc6} and \ref{propc3}, which is essentially the same as \cite[Proposition 6.5]{BB}. For convenience of readers, we reproduce the proof here.

\begin{proof}[Proof of Theorem \ref{thmc1}]
    Let $x_n\to x$, with $x_n\in A_n,n\geq 1$ and $x\in A$. By Proposition \ref{propc3}, there is a subsequence $n_l,l\geq 1$ and a stochastic process $(\mathbb{P}',X'_t)$ such that $\mathbb{P}'\big((X'_t)_{t\geq 0}\in \cdot\big)$ is a probability measure on $D(\mathbb{R}_+,B)$ and   $\mathbb{P}_{x_{n_l}}^{(n_l)}\big((X_t^{(n_l)})_{t\geq 0}\in \cdot\big)\Rightarrow \mathbb{P}'\big({(X'_t)}_{t\geq 0}\in \cdot\big)$. By using Proposition \ref{propc6} and (\ref{eqnc4}), (\ref{eqnc1}),  we have for any $\alpha>0$, $f\in C(B)$:
    \[\begin{aligned}
        &\int_0^\infty e^{-\alpha t}\mathbb{E}'[f(X'_t)]dt=\mathbb{E}'[\int_0^\infty e^{-\alpha t}f(X'_t)dt]=\lim_{l\to\infty}\mathbb{E}^{(n_l)}_{x_{n_l}}[\int_0^\infty e^{-\alpha t}f(X^{(n_l)}_t)dt]\\
        =&\lim_{l\to\infty}\int_{A_n}u_{\alpha,n_l}(x_{n_l},y)f(y)\mu_{n_l}(dy)=\int_A u_\alpha(x,y)f(y)\mu(dy)=\int_0^\infty e^{-\alpha t}\mathbb{E}_x[f(X_t)]dt.
    \end{aligned}\]
    Since both $X_t,X_t'$ are in $D(\mathbb{R}_+,B)$, and the fact that Laplace transformation determines the function a.e, we see that $\mathbb{E}_x[f(X_t)]=\mathbb{E}'[f(X'_t)]$ for each $t> 0$ and $f\in C(B)$. Hence,
    \[\mathbb{P}'(X'_t\in\cdot)=\mathbb{P}_x(X_t\in \cdot),\text{ for } t>0.\]
    In other words, $p^{(n_l)}_t(x_{n_l},\cdot)\Rightarrow p_t(x,\cdot)$. Since the argument works for any subsequence, and the space of probability measures on $K$ forms a separable complete metric space with Prohorov metric (see  \cite[Theorem 1.7, Chapter 3]{EK}), we have $p^{(n)}_t(x_n,\cdot)\Rightarrow p_t(x,\cdot)$.

    As an immediate consequence, we can check that $P_t^{(n)}f\rightarrowtail P_tf$ for every $t\geq0,f\in C(B)$ by using (\ref{eqnc1}). So, we can easily extend the result to finite dimensional cases by using the Markov property. For $t_1<t_2$ and $f,g\in C(B)$, we have
    \[\mathbb{E}_{x_n}^{(n)}(f(X^{(n)}_{t_1})g(X^{(n)}_{t_2}))=P^{(n)}_{t_1}\big(f\cdot P^{(n)}_{t_2-t_1}g\big)(x_n)\to P_{t_1}\big(f\cdot P_{t_2-t_1}g\big)(x)=\mathbb{E}_{x}\big(f(X_{t_1})g(X_{t_2})\big),\]
    where we use the fact $f\cdot P^{(n)}_{t_2-t_1}g\rightarrowtail f\cdot P_{t_2-t_1}g$, (\ref{eqnc4}) and (\ref{eqnc1}). The same argument works for longer sequences $t_1<t_2<\cdots<t_j$, so finite dimensional distributions converge. This, together with tightness (Proposition \ref{propc3}), implies Theorem \ref{thmc1} by   \cite[Theorem 7.8, Chapter 3]{EK}.
\end{proof}

Finally, noticing that by \cite[Theorem 4.5.1]{FOT}, the form $(\mcE,\mcF)$ is local if and only if the associated Hunt process has continuous sample paths (for q.e. starting points). In addition, by \cite[Theorem 10.2]{EK} the weak limit $X_t$ is a.s. continuous if $X^{(n)}_t$ is a.s. continuous for $n\geq 1$, so there is an immediate corollary of Theorem \ref{thmc1} concerning the local property. Readers can also find a purely analytic proof of a slightly stronger version in \cite[Theorem 2.14]{C}.

\begin{corollary}\label{coroc7}
    Assume all the basic settings. If $(\mcE_n,\mcF_n)$ is local for $n\geq 1$, then $(\mcE,\mcF)$ is also local.
\end{corollary}

\section{A trace theorem}\label{sec4}
Let $K$ be a $\USC$, and $(\mcE,\mcF)$ be a self-similar $D_4$-symmetric resistance form on $K$ with a renormalization factor $r$. For $n\geq 0$, let $$\partial_n K:=\bigcup_{w\in W_n}\Psi_w\partial_0 K.$$   In this section, we prove a trace theorem of  $(\mcE,\mcF)$
onto $\partial_n K$.

\vspace{0.2cm}
\noindent\textbf{Basic setting in Section \ref{sec4}.} We assume that the resistance metric $R$ associated with $(\mcE,\mcF)$ is jointly continuous on $K\times K$. Without loss of generality, we require that
$R(L_2,L_4)=1$. Note that $(\mcE,\mcF)$ is local by the self-similarity.  \vspace{0.2cm}

We call $A\subset K$ a \textit{simple set} if $A$ is in the form $A=\bigcup_{i=1}^{M_A}\Psi_{w^{(i)}}K$ with $M_A\in \mathbb{N}$ and $\mu(\Psi_{w^{(i)}}K\cap\Psi_{w^{(j)}}K)=0$ for $i\neq j$. On such a simple set $A$, we define
\begin{align*}
    \mcF_A&=\{f\in l(A):f\circ \Psi_{w^{(i)}}\in \mcF\hbox{ for every }1\leq i\leq M_A\},\\
    \mcE_A(f)&=\sum_{i=1}^{M_A} r^{-|w^{(i)}|}\mcE(f\circ \Psi_{w^{(i)}})\hbox{ for every }f\in \mcF_A.
\end{align*}
By self-similarity of $(\mcE,\mcF)$, we can check that the definition of $\mcF_A$ and $\mcE_A$ is independent of the decomposition of $A$. Moreover, if $\{A_n\}_{n=1}^M$ is a finite collection of simple sets such that $\bigcup_{n=1}^MA_n=K$ and $\mu(A_n\cap A_{n'})=0$ for $n\neq n'$, then
\begin{align*}
    \mcF&=\{f\in l(K):f|_{A_n}\in\mcF_{A_n}\hbox{ for every }1\leq n\leq M\},\\
    \mcE(f)&=\sum_{n=1}^{M} \mcE_{A_n}(f|_{A_n})\hbox{ for every }f\in \mcF.
\end{align*}\vspace{0.2cm}

Before stating the trace theorem, we introduce the concept of
\textit{Besov spaces} on the line segments, which naturally appear
as trace spaces of Dirichlet spaces. See \cite{CLST,HKtrace,J} for
related works.

\begin{definition}\label{def41}
    Let $\sigma>\frac12$. For each $u\in C[0,1]$, define
    \[
    [[u]]_{\Lambda_{2,2}^\sigma{[0,1]}}=\sqrt{\sum_{m=0}^\infty k^{(2\sigma-1)m}\sum_{l=0}^{k^m-1}\big(u(\frac{l}{k^m})-u(\frac{l+1}{k^m})\big)^2},
    \]
    and
    \[
    \Lambda_{2,2}^\sigma{[0,1]}=\big\{u\in C[0,1]:[[u]]_{\Lambda_{2,2}^{\sigma}[0,1]}<\infty\big\}.
    \]
    \end{definition}

    For a line segment $\overline{x,y}$, with $x\neq y$ in $\mathbb{R}^2$, we identify $\overline{x,y}$ with $[0,1]$ by the linear map $\gamma_{\overline{x,y}}:[0,1]\to \overline{x,y}$ defined as $\gamma_{\overline{x,y}}(t)=(1-t)x+ty$ for $t\in [0,1]$. Write  $|\overline{x,y}|$ the \textit{length} of $\overline{x,y}$.

    For $u\in C(\overline{x,y})$, we define
    \[
    [[u]]_{\Lambda_{2,2}^{\sigma}(\overline{x,y})}=|\overline{x,y}|^{{{1/2-\sigma}}}[[u\circ \gamma_{\overline{x,y}}]]_{\Lambda_{2,2}^{\sigma}[0,1]},
    \]
    and write $$\Lambda_{2,2}^{\sigma}(\overline{x,y})=\big\{u\in C(\overline{x,y}):[[u]]_{\Lambda_{2,2}^{\sigma}(\overline{x,y})}<\infty\big\}.$$ Notice that $[[u]]_{\Lambda_{2,2}^{\sigma}(\overline{x,y})}=[[u]]_{\Lambda_{2,2}^{\sigma}(\overline{y,x})}$.

Throughout this section, we always choose $\sigma=\sigma(r):=-\frac{\log r}{2\log k}+\frac{1}{2}$, so that
\[
[[u]]_{\Lambda_{2,2}^{\sigma(r)}[0,1]}=\sqrt{\sum_{m=0}^\infty r^{-m}\sum_{l=0}^{k^m-1}\big(u(\frac{l}{k^m})-u(\frac{l+1}{k^m})\big)^2},
\]
where $r\in (0,\infty)$ is the renormalization factor of $(\mcE,\mcF)$.
We remark that $\sigma(r)\geq\frac{3}{2}-\frac{\log N}{2\log k}$ by Lemma \ref{lemma410} to be proved later in this section.

Now we define the Besov space on $\partial_nK$ by taking Besov semi-norms to be the summation over  semi-norms on  all its line segment components.

\begin{definition}\label{def42}
    For $u\in C(\partial_n K)$, define
    \[[[u]]_{\Lambda_{2,2}^{\sigma(r)}(\partial_nK)}=\sqrt{\sum_{w\in W_n}\sum_{i=1}^4 [[u\circ \Psi_w]]^2_{\Lambda_{2,2}^{\sigma(r)}(\Psi_wL_i)}},\]
    and write $\Lambda_{2,2}^{\sigma(r)}(\partial_nK)=\big\{u\in C(\partial_nK):[[u]]_{\Lambda_{2,2}^{\sigma(r)}(\partial_nK)}<\infty\big\}$.
\end{definition}

The following is the main theorem in this section.
\begin{theorem}\label{thm43}
  We have $\mathcal{F}|_{\partial_nK}=\Lambda_{2,2}^{\sigma(r)}(\partial_nK)$, i.e.
    $$\{f|_{\partial_nK}: f\in\mcF\}=\Lambda_{2,2}^{\sigma(r)}(\partial_nK).$$

    In addition, there are $C_1,C_2>0$ depending only on $k$ such that
    \[C_1\cdot [[h|_{\partial_nK}]]^2_{\Lambda_{2,2}^{\sigma(r)}(\partial_nK)}\leq \mcE(h)\leq C_2\cdot [[h|_{\partial_nK}]]^2_{\Lambda_{2,2}^{\sigma(r)}(\partial_nK)}\quad \hbox{ for every }h\in \mathcal H_n,\]
    where
    \[\mathcal{H}_n=\big\{h\in \mcF:\mcE(h)=\inf\{\mcE(f):f|_{\partial_nK}=h|_{\partial_nK}, f\in\mcF\}\big\}\]
    is the space of functions that are harmonic in $K\setminus \partial_nK$.
\end{theorem}

\subsection{Building bricks and boundary graph}\label{sec41}
As a first step of the proof, we present an extension result related to the line segment $L_1$. Let $T_{m,K}=\bigcup_{w\in W_{m,1}}\Psi_wK$ for $m\geq 1$, and $B_{K}:=cl(T_{1,K}\setminus T_{2,K})$ the closure of $T_{1,K}\setminus T_{2,K}$ (in the metric $d$). Then, it is easy to see that $T_{1,K}$
\[T_{1,K}=\big(\bigcup_{w\in W_{*,1}}\Psi_wB_K\big)\bigcup L_1,\]
where we call each $\Psi_wB_K$ a level-$|w|$ \textit{building brick}.

We introduce two graphs based on $B_K$ and $T_{1,K}$.

\begin{definition}\label{def45}
    Let $K$ be a $\USC$.

    (a). Define $V_{B_K}=\{\Psi_1q_4,\Psi_kq_3\}\cup\{\Psi_{i1}q_4,\Psi_{ik}q_3\}_{i=1}^k$,
    \[E_{B_K}=\big\{\{\Psi_1q_4,\Psi_kq_3\},\{\Psi_1q_4,\Psi_{11}q_4\},\{\Psi_kq_3,\Psi_{kk}q_3\}\big\}\bigcup \big(\{\Psi_{i1}q_4,\Psi_{ik}q_3\}_{i=1}^k\big),\]
    and denote $G_{B_K}=(V_{B_K},E_{B_K})$ the associated finite graph.

    (b). Define $V_K=\bigcup_{w\in W_{*,1}} \Psi_wV_{B_K}$ and $E_K=\bigcup_{w\in W_{*,1}}\Psi_wE_{B_K}$, where $\Psi_wE_{B_K}:=\big\{\{\Psi_wx,\Psi_wy\}:\{x,y\}\in E_{B_K}\big\}$ for each $w\in W_{*,1}$. Denote $G_K=(V_K,E_K)$ the associated infinite graph.

    See Figure \ref{figure4} for an illustration.
\end{definition}

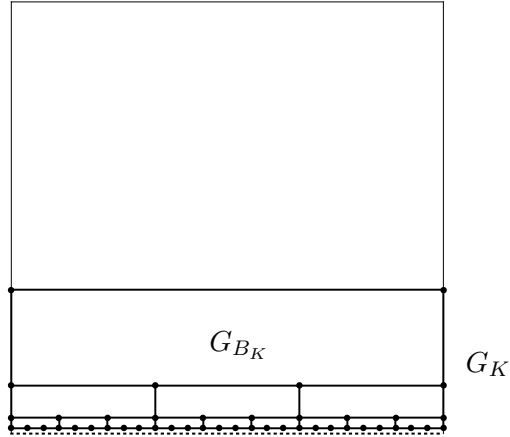
\begin{figure}[htp]
    \includegraphics[width=6cm]{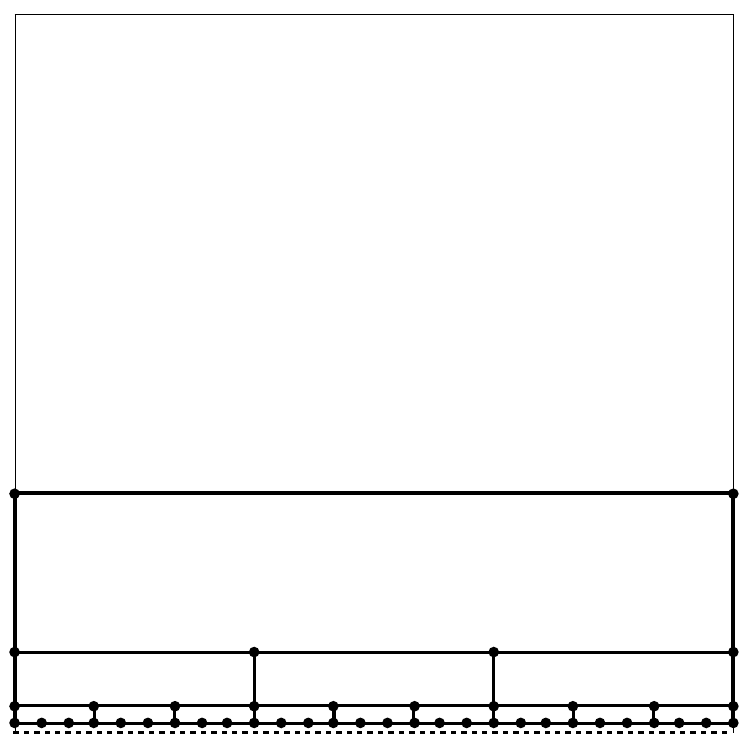}
    \begin{picture}(0,0)
        \put(-96,35){$G_{B_K}$}\put(1,27){$G_{K}$}
    \end{picture}
    \caption{The infinite graph $G_K$.}
    \label{figure4}
\end{figure}

The points in $V_{B_K}$ locate on the boundary of $B_K$, while the set $V_K$ consists of boundary vertices of all building bricks, with the closure of $V_k$ being $V_k\cup L_1$.

As before, we use the $\mathbb{R}^2$-coordinate to represent points in $V_K$.  One can check that
\[V_{B_K}=\big\{(0,\frac{1}{k}),(1,\frac{1}{k})\big\}\bigcup \big\{(\frac{l}{k},\frac{1}{k^2}):\,{l\in\mathbb{Z},\,}0\leq l\leq k\big\},\]
and
\[V_K=\bigcup_{m=0}^\infty\big\{(\frac{l}{k^m},\frac{1}{k^{m+1}}):\,{l\in\mathbb{Z},\,}0\leq l\leq k^m\big\}.\]

In the following, we define an energy on $G_K$.

\begin{definition}\label{def46}
    (a). For $f\in l(V_{B_K})$,  define
    \[\mcD_{G_{B_K}}(f)=\sum_{\{x,y\}\in E_{B_K}}\big(f(x)-f(y)\big)^2.\]

    (b). For $f\in l(V_K)$,  define
    \[\mcD_{G_K}(f)=\sum_{w\in W_{*,1}}r^{-|w|}\mcD_{G_{B_K}}(f\circ \Psi_w),\]
    where $r$ is the renormalization factor of $(\mcE,\mcF)$.
\end{definition}

\begin{definition}\label{def47}
    (a). Let $h\in \mcF$ be the function such that $0\leq h\leq 1$, $h|_{L_4}=0$, $h|_{L_2}=1$ and $\mcE(h)=1$; define $h'\in \mcF$ by
   \[
    h'(\Psi_wx)=\begin{cases}
       h\circ \Gamma_{r_3}(x),&\text{ if } w=ij\in W_2, i\in W_{1,1}, j\in W_{1,3},\\
      1,&\text{ if }w=ij\in W_2, i\notin W_{1,1},\\
        0,&\text{ elsewhere, }
    \end{cases}
    \]
    so that $0\leq h'\leq 1$, $h'|_{K\setminus T_{1,K}}=1$, $h'|_{T_{2,K}}=0$ and $\mcE(h')=k^2r^{-2}$.

    (b). For each $f\in l(V_{B_K})$, we define $\mcB_f^{(1)}, \mcB_f^{(2)}\in C(B_K)$ as
    \[\mathcal{B}^{(1)}_f(x)=f(\Psi_1q_4)+\big(f(\Psi_kq_3)-f(\Psi_1q_4)\big)\cdot h(x)\quad \hbox{ for }x\in B_K;\]
    \[
    \mathcal{B}^{(2)}_f(\Psi_ix)=f(\Psi_{i1}q_4)+\big(f(\Psi_{ik}q_3)-f(\Psi_{i1}q_4)\big)\cdot h(x)\quad\hbox{ for }1\leq i\leq k,\, x\in cl(K\setminus T_{1,K}),
    \]
    and define
    $\mathcal{B}_f\in C(B_K)$ as
    \[\mathcal{B}_f(x)=\mathcal{B}^{(1)}_f(x)\cdot h'(x)+\mathcal{B}^{(2)}_f(x)\cdot \big(1-h'(x)\big)\quad\hbox{ for }x\in B_K.\]
    Call $\mcB_f$ a \emph{building brick function} on $B_K$ induced by $f\in l(V_{B_K})$.
\end{definition}

\begin{lemma}\label{lemma48}
    Let $f,f'\in l(V_{B_K})$, and $\mathcal{B}_f$, $\mathcal{B}_{f'}$ be the functions defined in Definition \ref{def47}.

    (a). For $x\in \overline{\Psi_1q_4,\Psi_kq_3}$, \[\mcB_f(x)=f(\Psi_1q_4)\cdot\big(1-h(x)\big)+f(\Psi_kq_3)\cdot h(x).\]
    In addition, if $f\equiv c$ on $V_{B_K}$, then $\mcB_f\equiv c$ on $B_K$.

    (b). $\mcB_f\in \mcF_{B_K}$ and there exists $C>0$ depending only on $k$ and $r$ so that
    \[\mcE_{B_K}(\mcB_f)\leq C\cdot\mcD_{G_{B_K}}(f).\]

    (c). If $f(\Psi_kq_3)=f'(\Psi_1q_4)$ and $f(\Psi_{kk}q_3)=f'(\Psi_{11}q_4)$, then
    \[\mcB_f(1,x_2)=\mcB_{f'}(0,x_2)\quad\hbox{ for every }x_2\in [\frac{1}{k^2},\frac{1}{k}].\]

    (d). If $f(\Psi_1q_4)=f'(\Psi_{i1}q_4)$ and $f(\Psi_kq_3)=f'(\Psi_{ik}q_3)$ for some $1\leq i\leq k$, then
    \[\mcB_{f'}\circ \Psi_i(x)=\mcB_f(x)\quad\hbox{ for every } x\in \overline{\Psi_1q_4,\Psi_kq_3}.\]
\end{lemma}

\begin{proof}
    (a) is immediate from the construction of $\mathcal{B}_f$.

    (b). By the strongly local property of $(\mcE,\mcF)$, it is easy to see that $\mathcal{B}_f^{(1)}\in\mcF_{B_K},\,\mathcal{B}_f^{(2)}\in\mcF_{B_K}$ and
    \begin{align*}
        \mcE_{B_K}(\mathcal{B}_f^{(1)})&\leq \big(f(\Psi_1q_4)-f(\Psi_kq_3)\big)^2\cdot \mcE(h)=\big(f(\Psi_1q_4)-f(\Psi_kq_3)\big)^2,\\
        \mcE_{B_K}(\mathcal{B}_f^{(2)})&\leq r^{-1}\cdot\sum_{i=1}^k\big(f(\Psi_{i1}q_4)-f(\Psi_{ik}q_3)\big)^2\cdot \mcE(h)=r^{-1}\cdot\sum_{i=1}^k\big(f(\Psi_{i1}q_4)-f(\Psi_{ik}q_3)\big)^2.
    \end{align*}
    So by  \cite[Theorem 1.4.2]{FOT}, $\mathcal{B}_f\in \mcF_{B_K}$. To get the energy estimate of $\mathcal{B}_f$, without loss of generality, we assume $f(\Psi_1q_4)=0$ so that $\|\mcB^{(1)}_f\|_{L^\infty(B_K,\mu)}^2\leq C_1\cdot\mcD_{G_{B_K}}(f)$ and $\|\mcB^{(2)}_f\|_{L^\infty(B_K,\mu)}^2\leq C_1\cdot\mcD_{G_{B_K}}(f)$ for some $C_1>0$ depending only on $k$ and $r$. Then, the desired estimate of $\mcE_{B_K}(\mcB_f)$ follows from the inequality that $\mcE_{B_K}(g\cdot g')\leq 2\|g\|_{L^\infty(B_K,\mu)}^2\mcE_{B_K}(g')+2\|g'\|_{L^\infty(B_K,\mu)}^2\mcE_{B_K}(g)$ for every $g,g'\in \mcF_{B_K}$.

    (c). For $x_2\in[\frac {1}{k^2}, \frac 1k]$, one can check that \[\mcB_f(1,x_2)=f(\Psi_kq_3)\cdot h'(1,x_2)+f(\Psi_{kk}q_3)\cdot \big(1-h'(1,x_2)\big),\]
    while
    \[\mcB_{f'}(0,x_2)=f'(\Psi_1q_4)\cdot h'(0,x_2)+f'(\Psi_{11}q_4)\cdot \big(1-h'(0,x_2)\big).\]
    Noticing that $h'(1,x_2)=h'(0,x_2)$ by the $\Gamma_h$-symmetry of $h'$, we immediately have $\mcB_f(1,x_2)=\mcB_{f'}(0,x_2)$.

    (d) can be verified in a same way as (c).
\end{proof}

\begin{proposition}\label{lemma49}
    Let $f\in C(V_{K}\cup L_1)$ with $\mcD_{G_{K}}(f)<\infty$, and define $g\in C(T_{1,K}\setminus L_1)$ as
    \[g(x)=\mcB_{f\circ \Psi_w}\circ \Psi_w^{-1}(x),\text{ if }x\in \Psi_wB_K, w\in W_{*,1}.\]
    Then, $g$ extends continuously to $g\in \mcF_{T_{1,K}}$. In addition, there is $C>0$ depending only on $k$ and $r$, such that
    \[\mcE_{T_{1,K}}(g)\leq C\cdot\mcD_{G_K}(f).\]
\end{proposition}
\begin{proof}
    We extend $g$ to a function in $ C(T_{1,K})$ by letting $g|_{L_1}=f|_{L_1}$, noticing that by definition $\min_{p\in V_{B_K}}f\circ \Psi_w(p)\leq g(x)\leq \max_{p\in V_{B_K}}f\circ\Psi_w(p)$ for each $w\in W_{*,1}$ and $x\in \Psi_wB_K$.

    Let $h$ be the function defined in Definition \ref{def47}. For $m\geq 2$, define $g_m$ on $T_{1,K}$ by
    \[
    g_m(x)=
    \begin{cases}
        g(x),&\text{if }x\in T_{1,K}\setminus T_{m,K},\\
        f(\Psi_{w1}q_4)\cdot\big(1-h(\Psi^{-1}_wx)\big)+f(\Psi_{wk}q_3)\cdot h(\Psi_w^{-1}x),&\text{if }x\in \Psi_wT_{1,K}, \text{ for some }
        w\in W_{m-1,1}.
    \end{cases}
    \]
    It is easy to see that $g_m\in \mcF_{T_{1,K}}$  and
    \[\begin{aligned}
        \mcE_{T_{1,K}}(g_m)=&\mcE_{T_{1,K}\setminus T_{m,K}}(g_m)+\mcE_{T_mK}(g_m)\\=&\sum_{n=0}^{m-2}\sum_{w\in W_n}\mcE_{\Psi_wB_K}(\mcB_{f\circ \Psi_w}\circ\Psi_w^{-1})+\sum_{w\in W_{m-1,1}}\mcE_{\Psi_w T_{1,K}}(g_m)\\
        \leq&C_1\cdot \sum_{n=0}^{m-2}r^{-n}\mcD_{G_{B_K}}(f\circ \Psi_w)+r^{-m+1}\sum_{w\in W_{m-1,1}}\big(f(\Psi_{w1}q_4)-f(\Psi_{wk}q_3)\big)^2\cdot \mcE(h)\\
        \leq& C\cdot\mcD_{G_K}(f),
    \end{aligned}\]
    for some $C_1, C>0$ depending only on $k$ and $r$, where the first inequality follows from Lemma \ref{lemma48} (b), and the second inequality follows from the definition of $\mcD_{G_K}(f)$ and $\mcE(h)=1$. Clearly, $\{\|g_m\|_{L^2(T_{1,K},\mu)}\}$ is uniformly bounded, so we can find a subsequence of $\{g_{m}\}$ and $g\in\mcF_{T_{1,K}}$ such that $g_m\rightrightarrows g$ and $g_m\circ\Psi_i$ converges to  $g\circ\Psi_i$ weakly in $(\mcE,\mcF)$ for each $i\in W_{1,1}$. Thus, the desired energy estimate holds.
\end{proof}

\subsection{An upper bound estimate of the resistance metric}\label{sec42}
In this subsection, we prove an upper bound estimate of $R(x,y)$.

\begin{lemma}\label{lemma410}
    $\frac2k\leq r\leq\frac{N}{k^2}\leq 1-\frac{1}{k^2}$.
\end{lemma}
\begin{proof}
    First, we let $A=\bigcup_{w\in W_{1,1}\cup W_{1,3}}\Psi_wK$. Let $h\in \mcF$ such that $h|_{L_4}=0$, $h|_{L_2}=1$ and $\mcE(h)=1$; also let $h'\in \mcF_A$ such that $h'|_{L_4\cap A}=0$, $h'|_{L_2\cap A}=1$ and $\mcE(h')=\min\{\mcE_A(f):\,f\in \mcF_A,\, f|_{L_4\cap A}=0,\, f|_{L_2\cap A}=1\}$. One can check that
    \[
    h'(\frac{i+x_1}{k},\frac{j+x_2}{k})=\frac{h(x_1,x_2)}{k}+\frac{i}{k}\quad\hbox{ for }(x_1,x_2)\in K,\,i\in\{0,1,\cdots,k-1\}\,\hbox{ and }j\in\{0,k-1\}.
    \]
    Hence,
    \[
    \mcE(h)\geq \mcE(h')=\sum_{i\in W_{1,1}\cup W_{1,3}}r^{-1}\mcE(h'\circ \Psi_i)=2kr^{-1}k^{-2}\mcE(h).
    \]
    This implies that $r\geq \frac2k$.

    Next, let $u:\square\to \mathbb{R}$ be defined by $u(x_1,x_2)=x_1$. Then, by Proposition \ref{lemma49} and the fact that $r\geq \frac2k$, we know that we can find $h_0\in\mathcal{H}_0$ such that  $h_0|_{\partial_0K}=u|_{\partial_0K}$, and find $h_1\in\mathcal{H}_1$ such that $h_1|_{\partial_1K}=u|_{\partial_1K}$. Then, by self-similarity, we see that
    \[
    \mcE(h_0)\leq\mcE(h_1)=\sum_{i=1}^Nr^{-1}\mcE(h_1\circ\Psi_i)=\sum_{i=1}^Nr^{-1}k^{-2}\mcE(h_0)=\frac{N}{rk^2}\mcE(h_0).
    \]
    This implies that $r\leq \frac{N}{k^2}$.
\end{proof}

\begin{lemma}\label{lemma411}
$R(x,y)\leq 2k^3\cdot R(q_1,q_2)$ for every $x,y\in \partial_0K$.
\end{lemma}
\begin{proof}
    For $x=(\sum_{m=1}^M\alpha_m k^{-m},0)$, with $M<\infty$ and $\alpha_m\in \{0,1,\cdots,k-1\}$ for each $m=1,2,\cdots,M$, by the triangle inequality, we have
    \[R(x,q_1)\leq \sum_{m=1}^M \alpha_mr^m R(q_1,q_2)\leq (k-1)\cdot R(q_1,q_2)\sum_{m=1}^\infty r^m\leq \frac{r(k-1)R(q_1,q_2)}{1-r}\leq (k^3-1)\cdot R(q_1,q_2),\]
    where in the last inequality we use $r\leq \frac{N}{k^2}\leq \frac{k^2-1}{k^2}$ due to Lemma \ref{lemma410}. Now, for $x=(x_1,x_2),y=(y_1,y_2)\in \partial_0K$ such that $x_1,x_2,y_1,y_2$ have finite $k$-adic expansion, the desired estimate of $R(x,y)$ follows immediately by using symmetry, by using the  triangle inequality, and  by inserting at most $2$ boundary vertices from $\{q_i\}_{i=1}^4$. Finally, the above estimate extends to general $x,y\in \partial_0K$ by continuity.
\end{proof}

\begin{proposition}\label{prop412}
    $R(q_1,q_2)\leq C$ for some $C>0$ depending only on $k$.
\end{proposition}
\begin{proof}
    Let $h$ be the unique function on $K$ such that
    \[h(q_1)=-1,h(q_2)=1, \mcE(h)=\frac{4}{R(q_1,q_2)},\]
    i.e. $h$ is harmonic in $K\setminus \{q_1,q_2\}$. It suffices to show that $\mcE(h)\geq C\cdot\mcE(h')$, where $C>0$ depends only on $k$ and $h'$ is the unique function such that
    \[h'|_{L_4}=0, h'|_{L_2}=1, \mcE(h')=\frac{1}{R(L_2,L_4)}=1.\]

    Choose the smallest positive integer $l\geq \frac{\log 32k^3}{\log k^2-\log (k^2-1)}\geq-\frac{\log 32k^3}{\log r}$. Then by Lemma \ref{lemma411},  for any $x\in \Psi_1^l\partial_0 K$, one can see that $R(q_1,x)\leq r^lR(q_1,\Psi_1^{-l}x)\leq \frac{1}{16}R(q_1,q_2)$, and so
    \[\big|h(q_1)-h(x)\big|^2\leq R(q_1,x)\mcE(h)\leq \frac{1}{4}.\]
    As a consequence $h|_{\Psi_1^l\partial_0K}\leq -\frac{1}{2}$, and by symmetry we also have $h|_{\Psi_k^l\partial_0K}\geq \frac{1}{2}$.

    Now, let $h''\in \mcF_{T_{l,K}}$ be defined as $h''(\Psi_wx)=\frac{e(w)}{k^l}+\frac{1}{k^l}h'\circ \Psi_w(x)-\frac{1}{2}$ for each $w\in {W}_{l,1},x\in K$, where $e(w)=\sum_{l'=1}^l(w_{l'}-1)k^{l-l'}$. It is easy to see that $h''|_{\Psi_1^lL_4}=-\frac 1 2$ and $h''|_{\Psi_k^lL_2}=\frac 1 2$. We will prove that
    \begin{equation}\label{eqn31}
        \mcE(h)\geq \mcE_{T_{l,K}}(h)\geq \mcE_{T_{l,K}}(h'')=k^{-l}r^{-l}\cdot\mcE(h'),
    \end{equation}
    where the first inequality and the last equality are trivial.

    To see the second inequality in (\ref{eqn31}), we claim that \[\mcE_{T_{l,K}}(h'')=\inf\{\mcE_{T_{l,K}}(g):g|_{\Psi_1^lL_4}=-\frac{1}{2}, g|_{\Psi_k^lL_2}=\frac{1}{2},g\in \mcF_{T_{l,K}}\},\] i.e. $h''$ is harmonic in $T_{l,K}\setminus \big(\Psi_1^lL_4\cup \Psi_k^lL_2\big)$.

    First, for each $w,w'\in  W_{l,1}$ with $e(w')-e(w)=1$, one can check $h''|_{\Psi_wK\cup \Psi_{w'}K}$ is harmonic in $(\Psi_wK\cup \Psi_{w'}K)\setminus (\Psi_wL_4\cup \Psi_{w'}L_2)$. In fact, if $g$ is the unique function in $\mcF_{\Psi_wK\cup \Psi_{w'}K}$ such that  $g|_{\Psi_wL_4}=h''|_{\Psi_wL_4}=\frac{e(w)}{k^l}-\frac{1}{2}$, $g|_{\Psi_{w'}L_2}=h''|_{\Psi_{w'}L_2}=\frac{e(w)+2}{k^l}-\frac12$ and $g$ is harmonic in $(\Psi_wK\cup \Psi_{w'}K)\setminus (\Psi_wL_4\cup \Psi_{w'}L_2)$, then we can see $g+\frac{1}{2}-\frac{e(w)+1}{k^l}$ is anti-symmetric with respect to the horizontal reflection about $\Psi_wL_2$, whence $g|_{\Psi_wL_2}=\frac{e(w)+1}{k^l}-\frac{1}{2}$. It follows $h''|_{\Psi_wK\cup \Psi_{w'}K}=g$, since both $h''$ and $g$ are harmonic in $\Psi_wK\setminus (\Psi_wL_2\cup \Psi_wL_4)$ and $\Psi_{w'}K\setminus (\Psi_{w'}L_2\cup \Psi_{w'}L_4)$ and taking the same boundary values.

    Second, for each $v\in  \mcF_{T_{l,K}}$ such that $v|_{\Psi_1^lL_4\cup \Psi_k^lL_2}=0$, it is easy to find a partition $v=\sum_{w\in {W}_{l,1}:0\leq e(w)\leq k^l-2}v_w\cdot 1_{\Psi_wK\cup \Psi_{w'}K}$ where for each $0\leq e(w)\leq k^l-2$ with $e(w')-e(w)=1$, $v_w\in \mcF_{\Psi_wK\cup \Psi_{w'}K}\cap C(\Psi_wK\cup \Psi_{w'}K)$ and $v_w|_{\Psi_wL_4\cup \Psi_{w'}L_2}=0$. Thus, $\mcE_{T_{1,K}}(h'',v)=\sum_{w}\mcE_{T_{1,K}}(h'',v_w\cdot 1_{\Psi_wK\cup\Psi_{w'}K})=\sum_w\mcE_{\Psi_wK\cup\Psi_{w'}K}(h'',v_w)=0$, where the last equality holds since $h''|_{\Psi_wK\cup \Psi_{w'}K}$ is harmonic in $(\Psi_wK\cup \Psi_{w'}K)\setminus (\Psi_wL_4\cup \Psi_{w'}L_2)$ as showed in the last paragraph. This implies that $h''$ is harmonic in $T_{l,K}\setminus \big(\Psi_1^lL_4\cup \Psi_k^lL_2\big)$.

    Hence the claim holds, and thus (\ref{eqn31}) holds since $h|_{\Psi_1^l\partial_0K}\leq -\frac 12$ and $h|_{\Psi_k^l\partial_0K}\geq \frac 12$. So $\mcE(h)\geq k^{-l}r^{-l}\cdot\mcE(h')\geq \big(\frac{k}{k^2-1}\big)^{l}\cdot \mcE(h')$ since $r\leq \frac{N}{k^2}\leq \frac{k^2-1}{k^2}$. The proof is completed.
\end{proof}

\subsection{A restriction result} \label{sec43}
In this subsection, we prove a restriction result about the map $f\to f|_{L_1}$.

First, as an immediate application of Proposition \ref{prop412}, we have the following estimate.

\begin{lemma}\label{lemma413}
    There is $C>0$ depending only on $k$ such that

    (a). $\mcD_{G_{B_K}}(f)\leq C\cdot\mcE_{B_K}(f)$ for every $f\in \mcF_{B_K}$,

    (b). $\mcD_{G_K}(f)\leq C\cdot\mcE(f)$ for every $f\in \mcF$.
\end{lemma}

For a function $u$ on $L_1$, we introduce the notation
\[D_n(u)=\sqrt{\sum_{l=0}^{k^n-1}\big(u(\frac{l}{k^n},0)-u(\frac{l+1}{k^n},0)\big)^2}.\]
Then, $[[u]]_{\Lambda_{2,2}^{\sigma(r)}(L_1)}=\big\|r^{-n/2}D_n(u)\big\|_{l^2}=\sqrt{\sum_{n=0}^\infty r^{-n}|D_n(u)|^2}$.

Also, for $f\in C\big(cl(V_K)\big)$, we write
\[\tilde{D}_n(f)=\sqrt{\sum_{w\in W_{n,1}}\mcD_{G_{B_K}}(f\circ F_w)}.\]
Then, $\mcD_{G_K}(f)=\sum_{n=0}^{\infty} r^{-n}|\tilde{D}_n(f)|^2$.

\begin{lemma}\label{lemma414}
    Let $f\in C\big(cl(V_K)\big)$, then we have
    \[D_n(f|_{L_1})\leq 3\cdot\sum_{m=n}^\infty\tilde{D}_m(f),\quad\text{ for all }n\geq 0.\]
\end{lemma}
\begin{proof}
    Clearly, for $0\leq l\leq k^n$, we have
    \[
    f(\frac{l}{k^n},0)=f(\frac{l}{k^n},\frac{1}{k^{n+1}})+\sum_{m=n}^\infty \big(f(\frac{l}{k^n},\frac{1}{k^{m+2}})-f(\frac{l}{k^n},\frac{1}{k^{m+1}})\big).
    \]
    Thus by Minkowski inequality,
    \[
    \begin{aligned}
        D_n(f|_{L_1})&=\sqrt{\sum_{l=0}^{k^n-1}\big(f(\frac{l}{k^n},0)-f(\frac{l+1}{k^n},0)\big)^2}\\
        &=\sqrt{\sum_{l=0}^{k^n-1}\Big(f(\frac{l}{k^n},\frac{1}{k^{n+1}})-f(\frac{l+1}{k^n},\frac{1}{k^{n+1}})+\sum_{l'=0,1}(-1)^{l'}\sum_{m=n}^\infty \big(f(\frac{l+l'}{k^n},\frac{1}{k^{m+2}})-f(\frac{l+l'}{k^n},\frac{1}{k^{m+1}})\big)\Big)^2}\\
        &\leq\sqrt{\sum_{l=0}^{k^n-1}\big(f(\frac{l}{k^n},\frac{1}{k^{n+1}})-f(\frac{l+1}{k^n},\frac{1}{k^{n+1}})\big)^2}
        +\sum_{m=n}^{\infty}\sum_{l'=0,1}\sqrt{\sum_{l=0}^{k^n-1}\big(f(\frac{l+l'}{k^n},\frac{1}{k^{m+2}})-f(\frac{l+l'}{k^n},\frac{1}{k^{m+1}})\big)^2}\\
        &\leq 3\cdot\sum_{m=n}^\infty \tilde{D}_m(f).
    \end{aligned}
    \]
\end{proof}

Combining Lemma \ref{lemma413} and \ref{lemma414}, we can easily prove the following restriction theorem.
\begin{proposition}\label{thm415}
    There exists $C>0$ depending only on $k$ such that
    \[[[f|_{L_1}]]^2_{\Lambda_{2,2}^{\sigma(r)}(L_1)}\leq C\cdot\mcE(f) \quad\hbox{ for every } f\in \mcF.\]
\end{proposition}
\begin{proof}
    The proposition follows from the following inequality,
    \[\begin{aligned}
        \quad[[f|_{L_1}]]_{\Lambda_{2,2}^{\sigma(r)}(L_1)}&=\big\|r^{-n/2}D_n(f|_{L_1})\big\|_{l^2}\\
        &\leq 3\cdot\big\|\sum_{m=n}^\infty r^{-n/2}\tilde{D}_m(f)\big\|_{l^2}\\
        &=3\cdot\big\|\sum_{m=0}^\infty r^{m/2}r^{-(m+n)/2}\tilde{D}_{m+n}(f)\big\|_{l^2} \\
        &\leq 3\cdot\sum_{m=0}^\infty r^{m/2}\|r^{-(m+n)/2}\tilde{D}_{m+n}(f)\big\|_{l^2}\leq \frac{3}{1-r^{1/2}}\mcD^{1/2}_{G_K}(f)\leq C'\cdot\mcE^{1/2}(f),
    \end{aligned}\]
    where we use Lemma \ref{lemma414} in the second line and use Lemma \ref{lemma413} (b) in the last inequality. Here $C'$ depends only on $k$ noticing that $r\leq \frac{N}{k^2}$.
\end{proof}

\subsection{Proof of Theorem \ref{thm43}}
Finally, we finish this section with the proof of Theorem \ref{thm43}.

\begin{proof}[Proof of Theorem \ref{thm43}]
    We consider the $n=0$ case only, since the $n\geq 1$ case follows immediately from it with the self-similar property.

    First, by Proposition \ref{thm415}, one can see $\mathcal{H}_0|_{\partial_0K}\subset\Lambda_{2,2}^{\sigma(r)}(\partial_0K)$, and there is $C_1>0$ depending only on $k$ so that  for any $h\in \mathcal H_0$,
    \begin{equation}\label{eqn65add1}
        [[h|_{\partial_0K}]]^2_{\Lambda_{2,2}^{\sigma(r)}(\partial_0K)}\leq C_1\cdot \mcE(h).
    \end{equation}

    For the other direction, it suffices to show that for any $u\in \Lambda_{2,2}^{\sigma(r)}(\partial_0K)$, there is $f\in \mcF$ such that
    \[\begin{cases}
        f|_{\partial_0 K}=u,\\
        \mcE(f)\leq C_2\cdot [[u]]^2_{\Lambda_{2,2}^{\sigma(r)}(\partial_0K)}
    \end{cases}\]
    for some $C_2>0$ depending only on $k$.
    To achieve this, we apply the construction in Proposition \ref{lemma49}. We take two steps to construct a desired $f$. \vspace{0.2cm}

    \textit{Step 1. We  find $f'\in \mcF$ such that $f'|_{L_1}=u|_{L_1}$, $f'|_{L_3}=u|_{L_3}$, and
        \begin{equation}\label{eqn65add2}
            \mcE(f')\leq C_3\cdot [[u]]^2_{\Lambda_{2,2}^{\sigma(r)}(\partial_0K)},
        \end{equation}
        for some $C_3>0$ depending only on $k$. }
    \vspace{0.2cm}

    To complete \textit{Step 1}, we first use Proposition \ref{lemma49} to construct $f'$ on $T_{1,K}$ and $\Gamma_vT_{1,K}$ \big(recall $\Gamma_v(x_1,x_2)=(x_1,1-x_2)$ as defined in (\ref{eqn21})\big) satisfying $f'|_{L_1}=u|_{L_1}$, $f'|_{L_3}=u|_{L_3}$. Then for the middle part, we take
    \[
    \begin{aligned}
        f'(x)=&u(q_1)\big(1-h(x)\big)h''(x)+u(q_2)h(x)h''(x)\\
        &+u(q_3)h(x)\big(1-h''(x)\big)+u(q_4)\big(1-h(x)\big)\big(1-h''(x)\big)\quad \hbox{ for }x\in K\setminus \big(T_{1,K}\cup \Gamma_vT_{1,K}\big),
    \end{aligned}
    \]
    where $h$ is the same function defined in Definintion \ref{def47}, which satisfies $h\in\mcF$, $h|_{L_4}=0$ and $h|_{L_2}=1$; for $h''$, as usual one can use scaled copies of $h\circ \Gamma_{r_1}$ to build $h''\in \mcF$ such that $h''|_{T_{1,K}}=1,h''|_{\Gamma_vT_{1,K}}=0$. Clearly, $f'\in\mcF$ and satisfies the desired energy estimate. \vspace{0.2cm}

    Now, let $u'=u-f'|_{\partial_0K}$. We have
    \begin{equation}\label{eqn65add3}
        [[u']]_{\Lambda_{2,2}^{\sigma(r)}(\partial_0K)}\leq [[u]]_{\Lambda_{2,2}^{\sigma(r)}(\partial_0K)}+[[f'|_{\partial_0K}]]_{\Lambda_{2,2}^{\sigma(r)}(\partial_0K)}\leq C_4\cdot[[u]]_{\Lambda_{2,2}^{\sigma(r)}(\partial_0K)}
    \end{equation}
    for some $C_4>0$ depending only on $k$,
    where in the last inequality we use a combination of (\ref{eqn65add1}) and (\ref{eqn65add2}). In additon, we see that
    \[u'|_{L_1\cup L_3}=0.\]

    \vspace{0.2cm}
    \textit{Step 2. We construct $f''\in \mcF$ such that $f''|_{\partial_0 K}=u'$ and $\mcE(f'')\leq C_3\cdot[[u']]^2_{\Lambda_{2,2}^{\sigma(r)}(\partial_0K)}$. }
    \vspace{0.2cm}

    This can be fulfilled with a same argument as Step 1. First, we construct $f''$ on $\Gamma_{d_1}T_{1,K}$ and $\Gamma_{d_2}T_{1,K}$ using Proposition \ref{lemma49}, so that $f''|_{L_2\cup L_4}=u'|_{L_2\cup L_4}$ \big(recall $\Gamma_{d_1},\Gamma_{d_2}$ in (\ref{eqn21})\big). Then we extend $f''$ to $K$ with $0$ continuously since $f''|_{\bigcup_{i\in W_{1,4}}\Psi_{i}L_2}=f''|_{\bigcup_{i\in W_{1,2}}\Psi_{i}L_4}=0$. The energy estimate  \begin{equation}\label{eqn65add4}
        \mcE(f'')\leq C_3\cdot [[u']]^2_{\Lambda_{2,2}^{\sigma(r)}(\partial_0K)}
    \end{equation}
    is same as in Step 1, and $f''|_{L_1\cup L_3}=0$ is guaranteed by the construction. \vspace{0.2cm}

    Finally, we take $f=f'+f''$. We then have $f|_{\partial_0K}=u$, and by combining (\ref{eqn65add2})-(\ref{eqn65add4}), $\mcE(f)\leq C_2\cdot [[u]]^2_{\Lambda_{2,2}^{\sigma(r)}(\partial_0 K)}$ for some $C_2>0$ depending only on $k$.
\end{proof}

\section{An estimate of effective resistances}\label{sec5}
In this section, we establish an estimate of the resistance metric $R(\cdot,\cdot)$. We assume a same basic setting as in Section \ref{sec4}.\vspace{0.2cm}

\noindent\textbf{{Basic Setting in Section \ref{sec5}.}}
Let $K$ be a $\USC$ and $(\mcE,\mcF)$ be a self-similar $D_4$-symmetric resistance form on $K$. We assume that the associated resistance metric $R$  is jointly continuous on $K\times K$, and without loss of generality,
$R(L_2,L_4)=1$. \vspace{0.2cm}

We consider the \textit{geodesic metric} $d_G$ on $K$ in this section, which is defined as
\[d_G(x,y)=\inf\big\{length(\gamma): \gamma:[0,1]\to K \text{ is a rectifiable curve, } \gamma(0)=x, \gamma(1)=y\big\}, \]
for all $x,y$ in $K$. We point out that $d_G\asymp d$, with the constant depending only on $K$.

\begin{lemma}\label{lemma51}
    There exists $C\in (0,\infty)$ depending only on $K$ (more precisely, on $c_0$ of (A3)) such that
    \[d(x,y)\leq d_G(x,y)\leq C\cdot d(x,y)\hbox{ for every } x,y\in K.\]
\end{lemma}
\begin{proof}
    We first observe that
    \begin{equation}\label{eqn51}
        d_G(x,\partial_0K)\leq C'\quad\hbox{ for every } x\in K,\text{ for some $C'>0$ depending only on $k$}.
    \end{equation}
    In fact, a very rough chain argument gives that $C'\leq 4Nk^{-1}+4Nk^{-2}+\cdots$, where the first term $4Nk^{-1}$ is greater than the total length of $\partial_1K$.

    Now, given $x,y\in K$, we choose $n=0\vee\min\{m\in \mathbb{Z}:d(x,y)<c_0k^{-m}\}$, where $c_0$ is the constant in (A3). Then, by (A3), we can find $w,w'',w'\in W_{n}$ so that $x\in \Psi_wK$, $y\in \Psi_{w'}K$ and $\Psi_wK\cap \Psi_{w''}K\neq \emptyset$, $\Psi_{w'}K\cap \Psi_{w''}K\neq \emptyset$. By (\ref{eqn51}), there is $z_x\in \Psi_w(\partial_0K)$ so that $d_G(x,z_x)\leq C'k^{-n}$; there is $z_y\in \Psi_{w'}(\partial_0K)$ so that $d_G(y,z_y)\leq C'k^{-n}$. Thus, $d_G(x,y)\leq d_G(x,z_x)+d_G(z_x,z_y)+d_G(z_y,y)\leq 2C'k^{-n}+12k^{-n}$. The lemma follows immediately by the choice of $n$.
\end{proof}

The following theorem is the main result in this section.

\begin{theorem}\label{thm52}
    Let $\theta=-\frac{\log r}{\log k}$. Then there are $C_1,C_2>0$ depending only on $k$ such that
    \[
    C_1\cdot d_G(x,y)^{\theta}\leq R(x,y)\leq C_2\cdot d_G(x,y)^{\theta}\quad\hbox{ for every } x,y\in K,
    \]
    \[
    C_1\cdot \rho^{\theta}\leq R\big(x,K\setminus B^{(G)}_\rho(x)\big)\leq C_2\cdot \rho^{\theta}\quad\hbox{ for every } x\in K, 0<\rho\leq 1,
    \]
    where $B^{(G)}_\rho(x)=\{y\in K:d_G(x,y)<\rho\}$.
\end{theorem}

\begin{proof}
    It suffices to show $R(x,y)\leq C_2\cdot d_G(x,y)^{\theta}$ and $C_1\cdot \rho^{\theta}\leq R\big(x,K\setminus B^{(G)}_\rho(x)\big)$ for some $C_1,C_2>0$ depending only on $k$.\vspace{0.2cm}

    \textit{Proof of ``$R(x,y)\leq C_2\cdot d_G(x,y)^{\theta}$''}. First, by Lemma \ref{lemma411} and Proposition \ref{prop412}, $R(x',y')\leq C_3$ for every $x',y'\in \partial_0K$, where $C_3>0$ depends only on $k$. Then, by a routine chain argument and an argument of extension by continuity as the proof of Lemma \ref{lemma411}, $R(x',y')\leq C_4$ for every $x',y'\in K$, where $C_4>0$ depends only on $k$.  As a consequence, $R(x,y)\leq C_4\cdot r^n$ if there is $w\in W_n$ such that $\{x,y\}\subset \Psi_wK$.

    Now, we consider $x,y$ in the statement. We choose $n$ such that $k^{-n-1}<d_G(x,y)\leq k^{-n}$, and $w,w'\in W_n$ such that $x\in \Psi_wK, y\in \Psi_{w'}K$. Noticing that $d(\cdot,\cdot)\leq d_G(\cdot,\cdot)$, there are at most $C_5=[\pi(1+\sqrt{2})^2]$ many $w''\in W_n$ such that $d_G(x,\Psi_{w''}K)\leq k^{-n}$, since $\Psi_{w''}\square\subset B\big(x,k^{-n}(1+\sqrt{2})\big)$ (a ball centered at $x$ with respect to $d$) for any such $w''$. Since the geodesic path between $x,y$ will only intersect cells $\Psi_{w''}K$ such that $d_G(x,\Psi_{w''}K)\leq d_G(x,y)\leq k^{-n}$, we can find a chain $w=w^{(1)},w^{(2)},\cdots, w^{(l)}=w'$ in $W_n$ such that
    \[
    l\leq C_5,\quad  \Psi_{w^{(l')}}K\cap \Psi_{w^{(l'-1)}}K\neq \emptyset\quad\hbox{ for every }1< l'\leq l.
    \]
    Thus, we get $R(x,y)\leq C_4C_5\cdot r^n\leq C_2\cdot d_G(x,y)^{\theta}$ for some $C_2>0$ depending only on $k$.\vspace{0.2cm}

    \textit{Proof of ``$C_1\cdot \rho^{\theta}\leq R\big(x,K\setminus B^{(G)}_\rho(x)\big)$''}. The proof is based on Theorem  \ref{thm43} and the observation (\ref{eqn51}). We choose $n$ so that $k^{-n}\leq  C_6^{-1}\cdot\frac{\rho}{4}<k^{-n+1}$, where $C_6=C'+2$ and $C'$ is the constant in (\ref{eqn51}). Define $u\in C(\partial_nK)$ to be
    \[u(y)=\big(1-\rho^{-1} d_G(x,y)\big)\vee 0\quad\hbox{ for every } y\in \partial_nK.\]
    Clearly, $u$ is Lipschitz on each edge $\Psi_wL_i$, $w\in W_n$, $i\in\{1,2,3,4\}$, i.e.  $|u(y)-u(y')|\leq \rho^{-1}d(y,y')$ for every $y,y'\in \Psi_wL_i$.

    Let $h\in \mathcal{H}_n$ such that $h|_{\partial_nK}=u$ as in Theorem \ref{thm43}. Then we have\vspace{0.2cm}

    (a). $h(x)\geq \frac{3}{4}$ and $h|_{K\setminus B^{(G)}_\rho(x)}\leq \frac{1}{4}$.

    (b). $h|_{\Psi_wK}=0$ for any $w\in W_n$ such that $\Psi_wK\cap B^{(G)}_\rho(x)=\emptyset$. In particular, $h|_{\Psi_wK}$ is nonzero on at most $C_7=\big[\pi (4C_6k+\sqrt{2})^2\big]+1$ many cells $\Psi_wK$ with $w\in W_n$.
    \vspace{0.2cm}

    In fact, let $x\in F_uK$ for some $u\in W_n$. Then for any $y\in\Psi_{u}(\partial_0K)$, by (\ref{eqn51}), $d_G(x,y)\leq C_6k^{-n}\leq \frac{\rho}{4}$. So $u(y)\geq \frac 34$, and it follows $h(x)\geq\frac 34$. On the other hand, For any $y\in K\setminus B_\rho^{(G)}(x)$, it follows that $y\in F_{u'}K$ for some $u'\in W_n$ and $d_G(x, F_{u'}\big(\partial_0K)\big)\geq d_G(x,y)-d_G\big(y,F_{u'}(\partial_0K)\big)-2k^{-n}\geq \rho-C_6k^{-n}\geq \frac{3\rho}{4}$. This gives that $h|_{K\setminus B_\rho^{(G)}(x)}\leq\frac14$. Thus (a) follows. (b) is trivial as each $\Psi_wK$ with $\Psi_wK\cap B_\rho^{(G)}(x)\neq\emptyset$ is contained in $B(x,\rho+\sqrt{2}k^{-n})$.

    Thus, by (b) and Theorem \ref{thm43},
    \[\mcE(h)\leq C_8\cdot \sum_{w\in W_n}r^{-n}[[(h\circ \Psi_w)|_{\partial_0 K}]]^2_{\Lambda_{2,2}^{\sigma(r)}(\partial_0 K)}\leq C_9\cdot r^{-n},\]
    for some $C_8,C_9>0$ depending only on $k$. This together with (a) yields the desired estimate.
\end{proof}

\section{Uniqueness}\label{sec6}
In this section, we prove the uniqueness of the self-similar
$D_4$-symmetric resistance forms on $\USC$'s. The proof
is an application of the heat kernel estimate. To be more precise,
let $\mu$ be the normalized Hausdorff measure on $K$, in
other words, $\mu$ the unique probability measure determined by
$\mu=\frac 1 N\sum_{i=1}^N\mu\circ \Psi_i^{-1}$, then $(\mcE,\mcF)$
becomes a regular Dirichlet form on $L^2(K,\mu)$. Let $p_t(x,y)$ be
the \textit{heat kernel} associated with the heat operator $P_t$ of
a Dirichlet form $(\mcE,\mcF)$, i.e. $\{P_t\}_{t>0}$ is the unique
strongly continuous semigroup on $L^2(K,\mu)$ such that
\[\mcE(f)=\lim\limits_{t\to 0}\frac{1}{t}\langle(1-P_t)f,f\rangle_{L^2(K,\mu)},\]
where $\langle f,g\rangle_{L^2(K,\mu)}=\int_K
f(x)g(x)\mu(dx)$. The following \textit{heat kernel estimate},
\begin{equation}\label{eqn61}
    \begin{aligned}
        c_1\cdot &t^{-d_H/d_W}\exp\big(-c_2\cdot (\frac{d(x,y)^{d_W}}{t})^{\frac{1}{d_W-1}}\big)\leq p_t(x,y)\\
        &\leq c_3\cdot t^{-d_H/d_W}\exp\big(-c_4\cdot (\frac{d(x,y)^{d_W}}{t})^{\frac{1}{d_W-1}}\big)\hbox{ for every } 0<t\leq 1, x,y\in K
    \end{aligned}
\end{equation}
holds for some constants $c_1,c_2,c_3,c_4>0$, where  $d_H=\frac{\log N}{\log k}$ is the Hausdorff dimension of $K$ with respect to $d$, $d_W=-\frac{\log r}{\log k}+d_H$ is named \textit{walk dimension} in various contents \cite{B,GHL}. For short, we will write $\theta=-\frac{\log r}{\log k}$ as in Theorem \ref{thm52}.

There have been many deep works on the relation of resistance estimate and heat kernel estimate in various setting of metric measure spaces \cite{B,BCK,FHK,HK,ki4}. In particular, by applying  \cite[Theorems 15.10 and 15.11]{ki4} by Kigami, (\ref{eqn61}) is an immediate consequence of  the resistance estimate in Theorem \ref{thm52}.
\begin{proposition}\label{prop62}
    Let $K$ be a $\USC$ and $(\mcE,\mcF)$ be a self-similar $D_4$-symmetric resistance form on $K$ such that the associated resistance metric $R$ is jointly continuous on $K\times K$ and $R(L_2,L_4)=1$. Let $\theta=-\frac{\log r}{\log k}$, $d_H=\frac{\log N}{\log k}$ and $d_W=\theta+d_H$. We have (\ref{eqn61}) holds for some $c_1,c_2,c_3,c_4>0$ depending only on $K$.
\end{proposition}

The rest of this subsection is just the same as that in \cite{BBKT}. First, we refer to \cite{GHL} for a characterization of $\mcF$ as a Besov space $B_{2,\infty}^{\sigma}(K)$.

For $\sigma>0$, define the \textit{Besov space} $B_{2,\infty}^{\sigma}(K)$ on $K$ as
\[
B_{2,\infty}^{\sigma}(K)=\big\{f\in L^2(K,\mu):[[f]]_{B_{2,\infty}^{\sigma}(K)}<\infty\big\},
\]
where
\[[[f]]_{B_{2,\infty}^{\sigma}(K)}=\sqrt{\sup_{0<\rho<1}\rho^{-2\sigma-d_H}\int_K\int_{B_\rho(x)}|f(x)-f(y)|^2\mu (dy)\mu(dx)}.\]
Notice that the Hausdorff dimension $d_H$ does not depend on the form $(\mcE,\mcF)$.

\begin{proposition}\cite[Theorems 4.2 and 4.6]{GHL}\label{prop63}
    Let $K$ be a $\USC$ and $(\mcE,\mcF)$ be a Dirichlet form on $L^2(K,\mu)$ satisfying (\ref{eqn61}).

    (a). $\mcF=B_{2,\infty}^{d_W/2}(K)$, and there are $C_1,C_2>0$ depending only on $c_1,c_2,c_3,c_4$ such that
    \[C_1\cdot [[f]]_{B_{2,\infty}^{d_W/2}(K)}\leq \mcE(f)\leq C_2\cdot [[f]]_{B_{2,\infty}^{d_W/2}(K)} \quad \text{ for every } f\in\mcF.\]

    (b). $d_W/2=\inf\big\{\sigma>0:B_{2,\infty}^\sigma(K)=Constants\big\}$.
\end{proposition}

As an immediate consequence of Propositions \ref{prop62} and \ref{prop63}, we see that self-similar $D_4$-symmetric resistance forms on $K$ are comparable with each other.

\begin{corollary}\label{coro64}
    Let $K$ be a $\USC$ and $(\mcE,\mcF),(\mcE',\mcF')$ be two self-similar $D_4$-symmetric resistance forms on $K$ such that the associated resistance metrics $R,R'$ are jointly continuous on $K\times K$ and $R(L_2,L_4)=1,R'(L_2,L_4)=1$. Then, we have $\mcF'=\mcF$, and there is a constant $C>0$ depending only on $K$ so that
    \[C^{-1}\cdot \mcE'(f)\leq \mcE(f)\leq C\cdot\mcE'(f)\hbox{ for every } f\in \mcF=\mcF'.\]
\end{corollary}

We finalize the proof of Theorem \ref{thm61} with an argument of \cite{BBKT}. The following Proposition \ref{prop65} is essentially the same as  \cite[Theorem 2.1]{BBKT}.

\begin{proposition}\cite[Theorem 2.1]{BBKT}\label{prop65}
    Let $K$ be a $\USC$, and $(\mcE,\mcF), (\mcE',\mcF)$ be two self-similar $D_4$-symmetric resistance forms on $K$ such that the associated resistance metrics $R,R'$ are jointly continuous on $K\times K$. If
    \[\mcE(f)\leq \mcE'(f)\quad\text{ for every } f\in \mcF,\]
    then for $\eta>0$, there exists $C>0$ such that $\big(C\big((1+\eta)\mcE'-\mcE\big),\mcF\big)$ is also a self-similar $D_4$-symmetric resistance form on $K$ such that the associated resistance metric is jointly continuous and the effective resistance between $L_2,L_4$ equals $1$.
\end{proposition}

\begin{proof}[Proof of Theorem \ref{thm61}]
    We prove the uniqueness by contradiction. Assume $(\mcE,\mcF)$ and $(\mcE',\mcF')$ be  two different self-similar $D_4$-symmetric resistance forms on $K$ such that the associated resistance metrics $R,R'$ are jointly continuous on $K\times K$ and $R(L_2,L_4)=R'(L_2,L_4)=1$. By Corollary \ref{coro64}, we know $\mcF=\mcF'$. In addition, letting
    \[\sup(\mcE'|\mcE)=\sup\big\{\frac{\mathcal{E}'(f)}{\mathcal{E}(f)}:f\in \mcF\setminus Constants\big\},\quad \inf(\mcE'|\mcE)=\inf\big\{\frac{\mathcal{E}'(f)}{\mathcal{E}(f)}:f\in \mcF\setminus Constants\big\},\]
    we know that
    \begin{equation}\label{eqn62}
        \frac{\sup(\mcE'|\mcE)}{\inf(\mcE'|\mcE)}\leq C^2,
    \end{equation}
    where $C$ is the same constant in Corollary \ref{coro64}. Now for $\eta>0$, let $\mcE''=(1+\eta)\mcE'-\inf(\mcE'|\mcE)\mcE$. By Proposition \ref{prop65}, and by choosing a suitable constant $C'>0$, one can see that $(C'\mcE'',\mcF)$ is also a self-similar $D_4$-symmetric resistance form on $K$ such that the associated resistance metrics $R''$ is jointly continuous on $K\times K$ and $R''(L_2,L_4)=1$. In addition,
    \[\frac{\sup(C'\mcE''|\mcE)}{\inf(C'\mcE''|\mcE)}=\frac{\sup(\mcE''|\mcE)}{\inf(\mcE''|\mcE)}=\frac{(1+\eta)\sup(\mcE'|\mcE)-\inf(\mcE'|\mcE)}{\eta \inf(\mcE'|\mcE)}> \frac{1}{\eta}\cdot\big(\frac{\sup(\mcE'|\mcE)}{\inf(\mcE'|\mcE)}-1\big).\]
    When $\eta$ is small,  this contradicts Corollary \ref{coro64}, since a same bound estimate as (\ref{eqn62}) with $\mcE'$ replaced by $C'\mcE''$ should hold by Corollary \ref{coro64}.
\end{proof}

\section{Weak convergence: sliding the $\USC$}\label{sec7}
By removing the constrain that squares living on grids, we are able to view the $\USC$ family (with same $k,N$) as a collection of moving fractals, instead of isolated ones in the classical setting. In this section, we no longer focus on one $\USC$, but let the $\USC$ slide around. This situation is essentially new compared with previous ones.\vspace{0.2cm}

\noindent\textbf{{Basic settings of Section \ref{sec7}.}} In this section, we assume that $K_n,n\geq 1$ (and also $K$) be $\USC$'s with the same $k,N$, and assume $K_n\rightarrow K$ in Hausdorff metric (as subsets of $\square$).

Let $(\mcE,\mcF)$ be the unique self-similar $D_4$-symmetric resistance form on $K$ such that the resistance metric $R$ is jointly continuous and $R(L_2,L_4)=1$; also for $n\geq 1$, let $(\mcE_n,\mcF_n)$ be the unique self-similar $D_4$-symmetric resistance form on $K_n$ such that the resistance metric $R_n$ is jointly continuous and $R_n(L_2,L_4)=1$.

Let $\mu$ be the normalized Hausdorff measure on $K$, and let $\mu_n$ be the normalized Hausdorff measure on $K_n$. 

Let $\bm{M}=(\Omega,\mathcal{M},X_t,\mathbb{P}_x)$ be the Hunt process associated with the Dirichlet form $(\mcE,\mcF)$ on $L^2(K,\mu)$, and let $\bm{M}_n=(\Omega,\mathcal{M}_n,X^{(n)}_t,\mathbb{P}^{(n)}_x)$ be the Hunt process associated with the Dirichlet form $(\mcE_n,\mcF_n)$ on $L^2(K_n,\mu_n)$.\vspace{0.2cm}

For convenience, and to shorten the statements of the theorem, we will always use the notations as follows:

a). let $\{\Psi_i\}_{i=1}^N$ be the iterated function system (i.f.s.) generating $K$, and let $\{\Psi_{n,i}\}_{i=1}^N$ be the i.f.s. generating $K_n$. In addition, for each $w=w_1w_2\cdots w_m\in W_*$, similarly to $\Psi_w$, we write $\Psi_{n,w}:=\Psi_{n,w_1}\circ\Psi_{n,w_2}\circ\cdots\circ\Psi_{n,w_m}$ for short;

b). let $d_G$ be the geodesic metric on $K$, and let $d_{G,n}$ be the geodesic metric on $K_n$.\vspace{0.2cm}

The following is the main result in this section.

\begin{theorem}\label{thm71}
    Assume basic settings of Section \ref{sec7}. Then (i) and (ii) are equivalent:

    (i). $R_n\rightarrowtail R$. Moreover, for any $x_n\to x$, with $x_n\in K_n,n\geq 1$ and $x\in K$, we have \[\mathbb{P}_{x_n}^{(n)}\big((X^{(n)}_t)_{t\geq 0}\in \cdot\big)\Rightarrow \mathbb{P}_x\big((X_t)_{t\geq 0}\in \cdot\big),\]
    where the weak convergence ``$\Rightarrow$'' is in the sense of probability measures on $D(\mathbb{R}_+,\square)$ (the space of cadlag processes on $D(\mathbb{R}_+,\square)$  equipped with the usual Skorohod $J1$-topology).

    (ii). $d_{G,n},n\geq1$ are equicontinuous, viewed as functions on $K_n\times K_n\subset \square^2\subset \mathbb{R}^4$, equipped with the Euclidean metric on $\mathbb{R}^4$.
\end{theorem}

\noindent\textbf{Remark 1.} By continuous embedding, the weak convergence holds on $D(\mathbb{R}_+,\mathbb{R}^2)$ as well.
\vspace{0.2cm}

\noindent\textbf{Remark 2.} Since probability distributions forms a complete separable metric space with Prohorov metric (on $D(\mathbb{R}_+,\square)$ in Skorohod $J1$-topology), one can replace the sequence $K_n,n\geq1$ with a continuous family $K_s,s\in [0,1]$, such that $K_s\rightarrow K_{s_0}$ with $s\to s_0\in [0,1]$, in the theorem.

\subsection{Basic geometric properties}\label{sec71}
As preparations, we provide some easy observations on the normalized Hausdorff measures and geodesic metrics in this subsection.

\begin{lemma}\label{lemma72}
    Let $K_n,n\geq 1$ and $K$ be $\USC$'s with the same $k,N$.

    (a). Assume that for each $1\leq i\leq N$, $\Psi_{n,i}$ converges to $\Psi_i$ (in the finite dimensional linear space of affine mappings), then $K_n\to K$.

    (b). Conversely, if $K_n\to K$, then by reordering $\Psi_{n,i},1\leq i\leq N$ for each $n\geq 1$, we have $\Psi_{n,i}$ converges to $\Psi_i$ for each $1\leq i\leq N$.
\end{lemma}
\begin{proof}
    (a). One can easily check that $\Psi_{n,w}\to \Psi_w$ and hence $\delta(\Psi_{n,w}\square,\Psi_{w}\square)\to 0$ as $n\to\infty$ for each $w\in W_*$. As a consequence,
    \[\delta\big(\bigcup_{w\in W_m}\Psi_{n,w}\square,\bigcup_{w\in W_m}\Psi_{w}\square\big)\to0\text{, as }n\to\infty.\]
    (a) then follows immediately, noticing that
    \[\delta(K,\bigcup_{w\in W_m}\Psi_w\square)\leq k^{-m},\text{ and }\delta(K_n,\bigcup_{w\in W_m}\Psi_{n,w}\square)\leq k^{-m}\quad\hbox{ for every } n\geq 1.\]

    (b). By choosing a subsequence $n_l,l\geq 1$, we have $\Psi_{n_l,j}$ converges for each $1\leq j\leq N$. Take the limit to be $\Psi_j',1\leq j\leq N$. Then one can check that $\{\Psi_j'\}_{j=1}^N$ is still an i.f.s. satisfying Definition \ref{def21}, noticing that $\Psi_{n_l,j}\square\to \Psi'_j\square$. In addition, we have $\bigcup_{j=1}^N\Psi_j\square=\bigcup_{j=1}^N\Psi'_j\square$, hence $\{\Psi_j'\}_{j=1}^N=\{\Psi_i\}_{i=1}^N$ since there is a unique way to divide $\bigcup_{j=1}^N\Psi'_j\square$ into $N$ disjoint squares with side length $k^{-1}$ (cut $\square$ from outer to middle to see this).

    Noticing that the above argument works for any subsequence, for each $1\leq i\leq N$, we can find a sequence $\{j(n,i)\}_{j\geq 1}$ so that $\Psi_{n,j(n,i)}\to \Psi_i$ (otherwise, for some subsequence $n_l$, we can not find a further subsequence so that $\Psi_{n_{l_{l'}},j}\to \Psi_i$ for some $j$ as shown in the previous discussion). In addition, for each $i\neq i'$, the sequence $\{\Psi_{n,j(n,i)}\}_{n\geq 1}$ and $\{\Psi_{n,j(n,i')}\}_{n\geq 1}$ will be eventually disjoint since they converge to different limits, so we can assume that all the sequences are disjoint. Thus, by replacing $\Psi_{n,i}$ with $\Psi_{n,j(n,i)}$, we get a proper reordering so that $\Psi_{n,i}\to \Psi_i$ for each $1\leq i\leq N$.
\end{proof}

\noindent{\textbf{Remark.}} Due to Lemma \ref{lemma72}, in the following context,  we will always assume $\Psi_{n,i}$ converges to $\Psi_i$ for $1\leq i\leq N$, whenever $K_n\to K$.

\begin{proposition}\label{prop73}
    Let $K_n,n\geq 1$ and $K$ be $\USC$'s with the same $k,N$. Assume $K_n\to K$. Then $\mu_n\Rightarrow \mu$, where ``$\Rightarrow$'' means weak convergence on $(\square,d)$.
\end{proposition}
\begin{proof}
    By Lemma \ref{lemma72} (b), by properly reordering $\{\Psi_{n,i}\}_{1\leq i\leq N}$, we can assume $\Psi_{n,i}\to \Psi_i$ as $n\to\infty$ (in the finite dimensional linear space of affine mappings). Then, for each $f\in C(\square)$, we have
    \[N^{-m}\sum_{w\in W_m}f(\Psi_wq_1)=\lim\limits_{n\to\infty}N^{-m}\sum_{w\in W_m}f(\Psi_{n,w}q_1).\]
    In addition, for any $m\geq 0$, we have
    \[\begin{cases}
        \big|\int_K f(x)\mu(dx)-N^{-m}\sum_{w\in W_m}f(\Psi_wq_1)\big|\leq Osc_f(k^{-m}),\\
        \big|\int_{K_n} f(x)\mu_n(dx)-N^{-m}\sum_{w\in W_m}f(\Psi_{n,w}q_1)\big|\leq Osc_f(k^{-m}),\quad\hbox{ for every } n\geq 1,
    \end{cases}\]
    where $Osc_f(\eta)=\sup\big\{\big|f(x)-f(y)\big|:x,y\in \square\text{, }d(x,y)\leq\eta\big\}$, which is continuous at $0$ since $f$ is uniformly continuous. The proposition follows easily from the above estimates.
\end{proof}

Finally, we consider the geodesic metrics. In general, the geodesic
metric $d_{G,n}$ on $K_n$ (or $d_{G}$ on $K$) is complicated, so
it's important to have an equivalent but easier characterization of
the equicontinuity of $d_{G,n}$. In the following, we show it's
enough to consider the geodesic metric on $\bigcup_{i=1}^N\Psi_{n,i}\square$ instead.

\begin{proposition}\label{prop74}
    Let $K_n,n\geq 1$ be $\USC$'s with the same $k,N$ and assume $K_n\to K$. For $n\geq 1$, let $\tilde{K}_n=\bigcup_{i=1}^N\Psi_{n,i}\square$, and let $\tilde{d}_{G,n}$ be the geodesic metric on $\tilde{K}_n$. Then, the following (i),(ii), (iii) are equivalent.

    (i). $d_{G,n},n\geq1$ are equicontinuous.

    (ii). There do not exist $x,y\in \partial\square$ and $i\neq j$ so that \[\lim\limits_{n\to\infty}\Psi_{n,i}x=\lim\limits_{n\to\infty}\Psi_{n,j}y\text{, and } \liminf_{n\to\infty}d_{G,n}(\Psi_{n,i}x,\Psi_{n,j}y)>0.\]

    (iii). There do not exist $x,y\in \partial\square$ and $i\neq j$ so that \[\lim\limits_{n\to\infty}\Psi_{n,i}x=\lim\limits_{n\to\infty}\Psi_{n,j}y\text{, and } \liminf_{n\to\infty}\tilde{d}_{G,n}(\Psi_{n,i}x,\Psi_{n,j}y)>0.\]
\end{proposition}
\begin{proof}
    (i)$\Rightarrow$(ii). Since $d_{G,n}$ are equicontinuous, for any $x_n,y_n\in K_n,n\geq 1$ such that $d(x_n,y_n)\to 0$, we have $d_{G,n}(x_n,y_n)\to 0$.

    (ii)$\Rightarrow$(iii). This is trivial since $d_{G,n}(\Psi_{n,i}x,\Psi_{n,j}y)\geq \tilde{d}_{G,n}(\Psi_{n,i}x,\Psi_{n,j}y)$ for every $1\leq i,j\leq N$ and $x,y\in \partial\square$.

    (iii)$\Rightarrow$(ii). Let $x,y\in \partial\square$ and $i\neq j$, and let $\gamma_n$ be a geodesic path connecting $\Psi_{n,i}x,\Psi_{n,j}y$ in $\tilde{K}_n$. Then, we can find a path $\gamma'_n$ in $\bigcup_{i'=1}^N \Psi_{i'}(\partial\square)\subset K_n$ connecting $\Psi_{n,i}x,\Psi_{n,j}y$ in $K_n$, by replacing line segments going through the interior of $\Psi_{n,i'}\square$ with line segments lying along $\Psi_{n,i'}(\partial\square)$. Then $length(\gamma'_n)\leq \sqrt{2}\cdot length(\gamma_n)$, so $\sqrt{2}\cdot\tilde{d}_{G,n}(\Psi_{n,i}x,\Psi_{n,j}y)\geq d_{G,n}(\Psi_{n,i}x,\Psi_{n,j}y)$. This observation gives (iii)$\Rightarrow$(ii).

    (ii)$\Rightarrow$(i). We assume by contradiction that $d_{G,n},n\geq 1$ are not equicontinuous. We will show there are $x,y\in \partial\square$, $i\neq j$, so that $\lim\limits_{n\to\infty}\Psi_{n,i}x=\lim\limits_{n\to\infty}\Psi_{n,j}y,$ and  $\liminf\limits_{n\to\infty}d_{G,n}(\Psi_{n,i}x,\Psi_{n,j}y) >0$, which contradicts (ii).

    Since  $d_{G,n},n\geq 1$ are not equicontinuous, we can find $x_l,y_l\in K_{n_l}, n_1<n_2<\cdots$, so that $d(x_l,y_l)\to 0$ and $\inf_{l\geq 1}d_{G,n_l}(x_l,y_l)>\eta$ for some $\eta>0$. By passing to a subsequence, by compactness, we can in addition assume that $x_l\to z$ for some $z\in K$, so $y_l\to z$ as well. Next, noticing that by (\ref{eqn51}), $x_l,y_l$ belong to different $m_0$-cells in $K_{n_l}$ for some $m_0$ depending only on  $\eta$. So by passing to a further subsequence, we can in addition assume that  $x_l\in \Psi_{n_l,w}K_{n_l}$ and $y_l\in \Psi_{n_l,w'}K_{n_l}$ for some $w\neq w'\in W_{m_0}$ and for any $l\geq 1$. Then, we choose $m<m_0,v\in W_m$ so that $w_1w_2\cdots w_m=w'_1w'_2\cdots w'_m=v$ and $w_{m+1}\neq w'_{m+1}$. Let $i=w_{m+1},j=w'_{m+1}$ and $x=\Psi_{i}^{-1}\Psi_v^{-1}z$, $y=\Psi_{j}^{-1}\Psi_v^{-1}z$.

    Clearly, $z$ is in the intersection of two $(m+1)$-cells, so $x,y\in \partial\square$. In addition, $\Psi_{n,i}x\to \Psi_v^{-1}z$, $\Psi_{n,j} y\to \Psi_v^{-1}z$ so $\lim_{n\to\infty}\Psi_{n,i}x=\lim_{n\to\infty}\Psi_{n,j}y$. Moreover, let $\tilde{x}_l=\Psi_{n_l,i}^{-1}\Psi_{n_l,v}^{-1}x_l$ and $\tilde{y}_l=\Psi_{n_l,j}^{-1}\Psi_{n_l,v}^{-1}y_l$, then $\tilde{x}_l\to x$, and it is easy to check that $d_{G,n_l}(\tilde{x}_l,x)\to 0$ by taking the advantage that $x\in \partial\square$. Similarly, $\tilde{y}_l\to y$ and $d_{G,n_l}(\tilde{y}_l,y)\to 0$. Noticing that $d_{G,n_l}(\Psi_{n_l,i}\tilde{x}_l,\Psi_{n_l,j}\tilde{y}_l)\geq k^{m}d_{G,n_l}(x_l,y_l)\geq k^{m}\eta$, we have $\liminf_{l\to\infty}d_{G,n_l}(\Psi_{n_l,i}x,\Psi_{n_l,j}y)\geq k^{m}\eta$.  Thus we have proved (ii)$\Rightarrow$(i).
\end{proof}

\begin{example}\label{example96}
    Let $k=7$ and $N=4(k-1)+8=32$. For $1\leq i\leq 24$, $\Psi_i$ is defined as usual. Define $\Psi_{25}(x)=\frac{1}{7}x+(z+\frac{2}{7},\frac{1}{7})$, with $0\leq z\leq \frac{1}{14}$. To satisfy the symmetry conditions, the contraction maps $\Psi_i$ for $25\leq i\leq 32$ are uniquely determined. We write $K(z)$ for the unique compact $K$ such that  $K=\bigcup_{i=1}^{32}\Psi_iK$, see Figure \ref{figure5}. We can see there are two cases.

    \emph{Case 1. $z\neq0$, $z\neq \frac{1}{14}$.} In this case, if $z_n\to z$, then the geodesic metric $d_{G,n}$ on $K(z_n)$, $n\geq 1$ are equicontinuous.

    \emph{Case 2. $z=0$ or $z=\frac{1}{14}$.} In this case, if $z_n\to z$, then the geodesic metric $d_{G,n}$ on $K(z_n)$, $n\geq 1$ are not equicontinuous (unless $z_n=z$ for all $n\geq n_0$ for some $n_0$).
\end{example}

\begin{figure}[htp]
    \includegraphics[width=4.9cm]{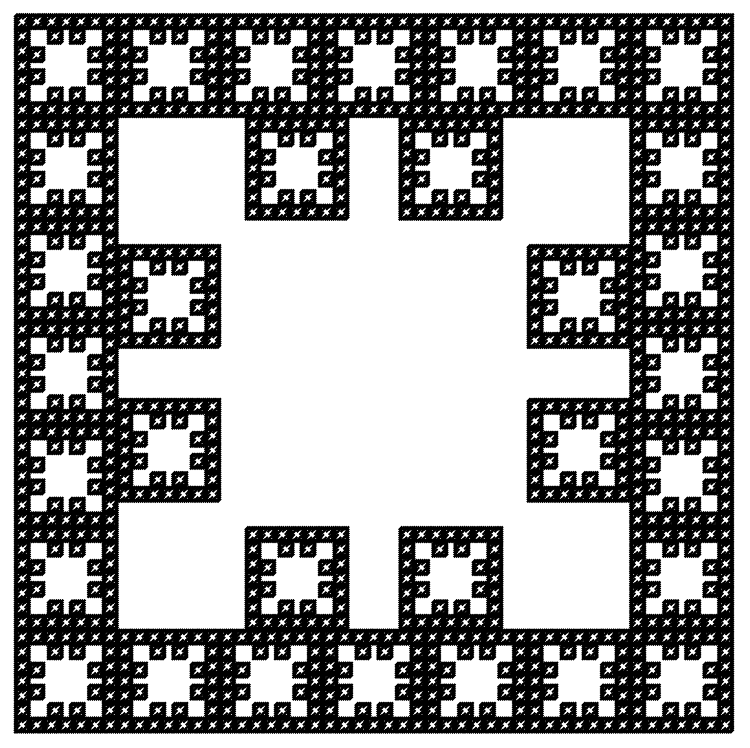}\quad
    \includegraphics[width=4.9cm]{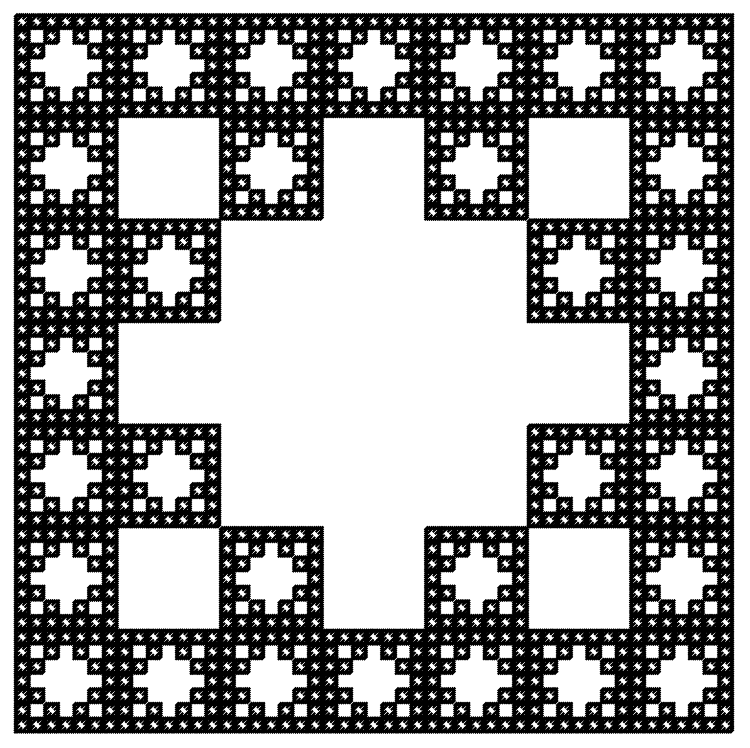}\quad
    \includegraphics[width=4.9cm]{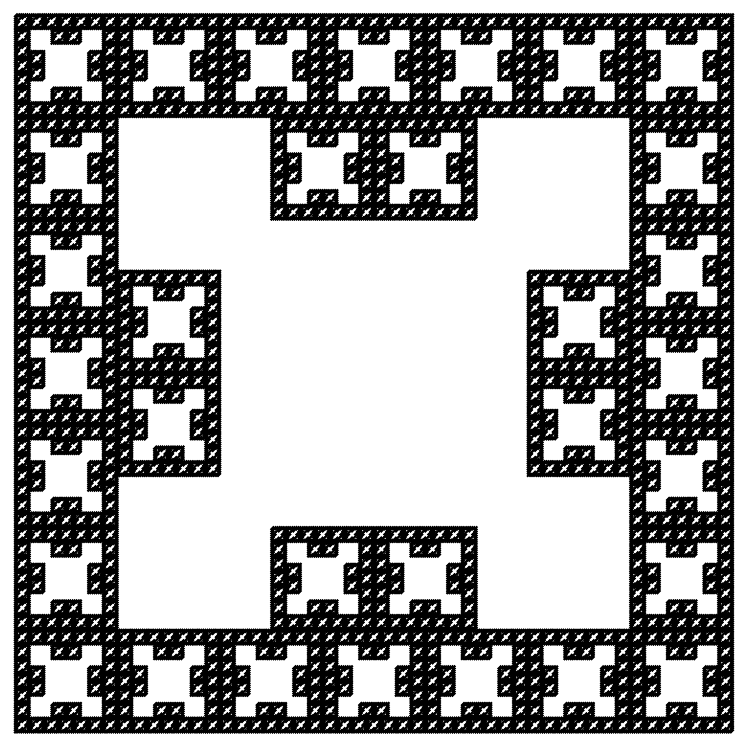}
    \begin{picture}(0,0)
        \put(-163,1.1){$z=\frac{1}{28}$}
        \put(-10,1.1){$z=0$}
        \put(140,1.1){$z=\frac{1}{14}$}
    \end{picture}
    \caption{The $\USC$ $K(z)$.}
    \label{figure5}
\end{figure}
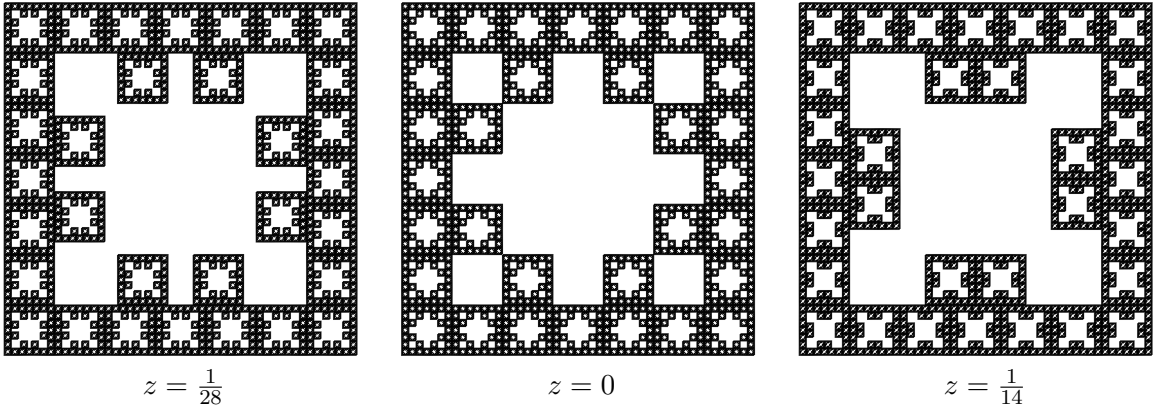

\begin{lemma}\label{lemma76}
    Let $K_n,n\geq 1$ and $K$ be $\USC$'s with the same $k,N$, and assume that $K_n\to K$ and $d_{G,n},n\geq 1$ are equicontinuous. Then, there exists $C>0$ depending on $K$ so that
    \[
    \limsup\limits_{n\to\infty}d_{G,n}(x_n,y_n)\leq C\,d_G(x,y)
    \]
    for any $x,y\in K$ and $x_n,y_n\in K_n,n\geq 1$ with $x_n\to x,y_n\to y$ as $n\to\infty$.
\end{lemma}

\begin{proof}
    If $x=y$, then $\lim\limits_{n\to\infty}d_{G,n}(x_n,y_n)=0$ by the equicontinuity assumption.

    Recall that $c_0$ is the constant appearing in (A3). If $d(x,y)\geq c_0k^{-1}$, then since $d_{G,n}(x_n,y_n)\leq 2C_1+2$ where $C_1$ is the constant $C'$  in \eqref{eqn51}, we have $\limsup\limits_{n\to\infty}d_{G}(x_n,y_n)\leq k(2C_1+2)c_0^{-1}\,d_G(x,y)$.

Finally, assume $d(x,y)\in [c_0k^{-m-1},c_0k^{-m})$ for some $m\geq
1$. By using (A3), we can find $w,w',w''\in W_m$ and
 $x'_n\in \Psi_{n,w}K_n, y'_n\in
\Psi_{n,w'}K_n$ for large enough $n$, such that  $x\in \Psi_wK, y\in
\Psi_{w'}K, \Psi_{w}K\cap \Psi_{w''}K\neq \emptyset, \Psi_{w'}K\cap
\Psi_{w''}K\neq\emptyset$  and $\lim_{n\to \infty}d(x'_n,x) =
\lim_{n\to \infty}d(y'_n,y) = 0$. Choose
    \[
    z\in \Psi_wK\cap \Psi_{w''}K\hbox{ and }z'\in \Psi_{w''}K\cap \Psi_{w'}K,
    \]
    and let
    \begin{align*}
        z_{n,w}=\Psi_{n,w}\circ\Psi_w^{-1}(z),&\ z_{n,w''}=\Psi_{n,w''}\circ \Psi_{w''}^{-1}(z),\\\ z'_{n,w''}=\Psi_{n,w''}\circ\Psi_{w''}^{-1}(z'),&\ z'_{n,w'}=\Psi_{n,w'}\circ \Psi_{w'}^{-1}(z').
    \end{align*}
    Then, by Lemma \ref{lemma72} (b) and the equicontinuity assumption, we know that
    \begin{equation}\label{eqn71}
\begin{aligned}
\lim\limits_{n\to\infty}d_{G,n}(x_n,x'_n)=0,& \
            \lim\limits_{n\to\infty}d_{G,n}(y_n,y'_n)=0,\\
\lim\limits_{n\to\infty}d_{G,n}(z_{n,w},z_{n,w''})=0,&\
            \lim\limits_{n\to\infty}d_{G,n}(z'_{n,w''},z'_{n,w''})=0.
\end{aligned}
    \end{equation}
    Hence,
    \begin{align*}
        &\limsup_{n\to\infty}d_{G,n}(x_n,y_n)\\
        \leq&\limsup_{n\to\infty}\big(d_{G,n}(x_n,x'_n) + d_{G,n}(x'_n,z_{n,w})+d_{G,n}(z_{n,w},z_{n,w''})\\ &\qquad\qquad+d_{G,n}(z_{n,w''},z'_{n,w''})+d_{G,n}(z'_{n,w''},z'_{n,w'})+d_{G,n}(z'_{n,w'},y'_n) + d_{G,n}(y'_n,y_n) \big)\\
        =&\limsup\limits_{n\to\infty}\big(d_{G,n}(x'_n,z_{n,w})+d_{G,n}(z_{n,w''},z'_{n,w''})+ d_{G_n}(z'_{n,w'},y'_n)\big)\\
        \leq &3(C_1+2)k^{-m}\\
        \leq &3(C_1+2)k^{-m}\,c_0^{-1}k^{m+1}\,d(x,y)\\
        \leq &3(C_1+2)kc_0^{-1}\, d_{G,n}(x,y),
    \end{align*}
    where we use \eqref{eqn71} in the equality,  use \eqref{eqn51} in the second inequality,  use the fact that $d(x,y)\geq c_0k^{-m-1}$ in the third inequality, and the last inequality is due to the fact that $d(x,y)\leq d_{G,n}(x,y)$.
\end{proof}

\subsection{Proof of Theorem \ref{thm71}}\label{sec72} We prove Theorem \ref{thm71} in this subsection.

\begin{lemma}\label{lemma77}
    Let $K_n,n\geq 1$ and $K$ be $\USC$'s with the same $k,N$, and assume that $K_n\to K$ and $d_{G,n},n\geq 1$ are equicontinuous. Then, for each subsequence, we can find a further subsequence $n_l$, $l\geq 1$, and a resistance metric $R_\infty$ on $K$ so that $R_{n_l}\rightarrowtail R_\infty$.

    As a consequence, $(\mcE_{n_l},\mcF_{n_l})$ $\Gamma$-converges to $(\mcE_\infty,\mcF_\infty)$ on $C(\square)$ as $l\to \infty$, where $(\mcE_\infty,\mcF_\infty)$ is the form associated with $R_\infty$. In addition, $(\mcE_\infty,\mcF_\infty)$ is strongly local and regular.

    Finally, there is $\frac2k\leq r_\infty\leq\frac{N}{k^2}$ so that $r_{n_l}\to r_\infty$, where $r_{n}$ is the renormalization factor of $(\mcE_n,\mcF_n)$. Moreover, there exist $C_1,C_2\in (0,\infty)$ so that
    \[
    C_1d(x,y)^{\theta_\infty}\leq R_\infty(x,y)\leq C_2d(x,y)^{\theta_\infty}\ \hbox{ for every }x,y\in K,\]
    where $\theta_\infty=-\frac{\log r_\infty}{\log k}$.
\end{lemma}
\begin{proof}
    For any $n\geq 1$ and $x,y\in K_n$, by using Theorem \ref{thm52} and noticing that $d(x,y)\leq d_{G,n}(x,y)$, we have
    \[C_3\cdot d^{\theta^*}(x,y)\leq R_n(x,y)\leq C_4\cdot d_{G,n}^{\theta_*}(x,y),\]
    with $\theta_n=-\frac{\log r_n}{\log k}$, $\theta^*=\sup_{n\geq 1}\theta_n\geq 2-\frac{\log N}{\log k}$ and $\theta_*=\inf_{n\geq 1}\theta_n<1$  (by Lemma \ref{lemma410}) for some constants $C_3,C_4>0$. Then, by the assumption  $d_{G,n},n\geq 1$ are equicontinuous, all the assumptions (uniform lower and upper bound estimate of $R_n,n\geq 1$) of Theorem \ref{thmb3} are satisfied, so one can find a further subsequence (of the subsequence in the statement of the lemma) such that $R_{n_l}\rightarrowtail R_\infty$, where $R_\infty$ is a resistance metric on $K$.

    Then, using Theorem \ref{thmb5}, we see $(\mcE_{n_l},\mcF_{n_l})$ $\Gamma$-converges  to $(\mcE_\infty,\mcF_\infty)$ on $C(\square)$, the form on $K$ associated with $R_\infty$. By Corollary \ref{coroc7}, we know that $(\mcE_\infty,\mcF_\infty)$ is local (hence strongly local since it's a resistance form on a compact space).

    Finally, by passing to a further subsequence if necessary, one has $r_{n_l}\to r_\infty$ for some $r_\infty$. Noticing that by Lemma \ref{lemma410} we always have $\frac{2}{k}\leq r_{n_l}\leq \frac{N}{k^2}$, so $r_\infty$ has the same bounds. Moreover, applying  Theorem \ref{thm52} to each $K_{n_l},l\geq 1$ and by taking limit, we can see that
    \begin{align*}
        &R_\infty(x,y)=\lim\limits_{l\to\infty}R_{n_l}(x_{n_l},y_{n_l})\geq C_5\lim_{l\to\infty}d(x_{n_l},y_{n_l})^{\theta_{n_l}}=C_5d(x,y)^{\theta_\infty},\\
        &R_\infty(x,y)=\lim\limits_{l\to\infty}R_{n_l}(x_{n_l},y_{n_l})\leq C_6\lim_{l\to\infty}d_{G,n_l}(x_{n_l},y_{n_l})^{\theta_{n_l}}\leq C_6\big(C_7d_G(x,y)\big)^{\theta_\infty},
    \end{align*}
    for any $x,y\in K$ and $x_{n_l},y_{n_l}\in K_{n_l}$, $l\geq 1$ such that $x_{n_l}\to x$, $y_{n_l}\to y$ as $l\to\infty$, where we also use Lemma \ref{lemma76} in the last inequality. The desired estimate of $R_\infty(x,y)$ then follows by using Lemma \ref{lemma51}.
\end{proof}

To prove Theorem \ref{thm71}, we need to verify $R_\infty=R$. For convenience,  in lemmas in the rest of this section, let us assume $R_n\rightarrowtail R_\infty$, $r_n\to r_\infty$, and write $(\mcE_\infty,\mcF_\infty)$ for the associated resistance form on $K$.

We will apply the uniqueness theorem, Theorem \ref{thm61}. The main difficulty is to verify that $(\mcE_\infty,\mcF_\infty)$ is a self-similar form. Nevertheless, it is not hard to firstly see a partial result.

\begin{lemma}\label{lemma78}
    Let $K_n,n\geq 1$ and $K$ be $\USC$'s with the same $k,N$. In addition, assume $R_n\rightarrowtail R_\infty$ for some resistance metric $R_\infty$ on $K$. Let $f\in C(K)$ and assume $f|_{\partial_1K}=0$, then  $\mcE_\infty(f)=\sum_{i=1}^N r_\infty^{-1}\mcE_\infty(f\circ \Psi_i)$. In particular, $f\in \mcF_\infty$ if and only if $f\circ \Psi_i\in \mcF_\infty$ for every $1\leq i\leq N$.
\end{lemma}
\begin{proof}
    By Theorem \ref{thmb5}, $\mcE_n$ $\Gamma$-converges  to $\mcE_\infty$ on $C(\square)$. So, there is a sequence $f_{n}\in C(K_n), n\geq 1$ such that $f_{n}\rightarrowtail f$ and $\mcE_\infty(f)=\lim_{n\to\infty}\mcE_{n}(f_{n})$. For each $1\leq i\leq N$, $f_{n}\circ \Psi_{n,i}\rightarrowtail f\circ \Psi_i$ since $x_n\rightarrow x$ implies $\Psi_{n,i}x_n\to \Psi_ix$ by Lemma \ref{lemma72} (b), so
    \[\mcE_\infty(f)=\lim_{n\to\infty}\mcE_{n}(f_{n})=\lim_{n\to\infty}\sum_{i=1}^N r_{n}^{-1}\mcE_{n}(f_{n}\circ \Psi_{n,i})\geq \sum_{i=1}^N r_\infty^{-1}\mcE_\infty(f\circ \Psi_i).\]

    To see the other direction, we use the condition $f|_{\partial_1 K}=0$. For each $1\leq i\leq N$, we pick a sequence $g_{n,i}\in C(K_n)$,  such that $g_{n,i}\rightarrowtail f\circ \Psi_i$ and $\mcE_\infty(f\circ \Psi_i)=\lim\limits_{n\to\infty}\mcE_{n}(g_{n,i})$. In addition, we can assume $g_{n,i}|_{\partial_0K}=0$, since otherwise, for each $1\leq i\leq N$, by Lemma \ref{lemmab2} (c):(c-i)$\Rightarrow$(c-iii) and by the Markov property, we can find a sequence of positive numbers $\varepsilon_{n,i}, n\geq 1$ that converges to $0$, such that $g_{n,i}-\big(g_{n,i}\vee (-\varepsilon_{n,i})\big)\wedge \varepsilon_{n,i}$ satisfies the desired requirement.  Then, for each $n$, we glue $g_{n,i}$ together: define $g_{n}\in C(K_n)$ by $g_{n}\circ \Psi_{n,i}=g_{n,i}$, $1\leq i\leq N$. Then,
    \[\sum_{i=1}^N r_\infty^{-1}\mcE_\infty(f\circ \Psi_i)=\lim\limits_{n\to\infty}\sum_{i=1}^Nr_{n}^{-1}\mcE_{n}(g_{n,i})=\lim\limits_{n\to\infty}\mcE_{n}(g_{n})\geq \mcE_\infty(f).\]
\end{proof}

Since the form $(\mcE_\infty,\mcF_\infty)$ is strongly local and regular, we can apply the tool of energy measure. For each function $f\in \mcF_\infty$, noticing that $f\in C(K)$ in our setting, we can define the \textit{energy measure} $\nu_f$ associated with $f$ to be the unique Radon measure on $K$ such that
\[\int_K g(x)\nu_f(dx)=2\mcE_\infty(fg,f)-\mcE_\infty(f^2,g)\quad\hbox{ for every } g\in\mcF_\infty.\]
In particular, it is well known that $\nu_f(K)=2\mcE_\infty(f)$ (see  \cite[Lemma 3.2.3]{FOT}). A useful fact is that for any Borel subset $A\subset K$,  $\nu_f(A)$  only depends on the value of $f$ on $A$.

\begin{lemma}\label{lemma79}
    Let $f,g\in \mcF_\infty$ and assume $f=g$ on a Borel subset $A$ in $K$, then $\nu_f(A)=\nu_g(A)$.
\end{lemma}
\begin{proof}
    We have the triangle inequality for the energy measures $\sqrt{\nu_f(A)}\leq \sqrt{\nu_g(A)}+\sqrt{\nu_{f-g}(A)}$ (see \cite[Page 123]{FOT}). Also,  $\nu_{h}\big(\{x\in K:h(x)=0\}\big)=0$ for any $h\in \mcF_\infty$ (see \cite[Lemma 2.7]{BBKT}). These two observations give that $\nu_f(A)\leq \nu_g(A)$. In a same way, $\nu_g(A)\leq \nu_f(A)$. The lemma follows.
\end{proof}

We can make Lemma \ref{lemma78} a little stronger with the language of the energy measure.

\begin{lemma}\label{lemma710}
    Let $f\in \mcF_\infty$, then for any Borel set $A\subset K\setminus \partial_1K$, we have
    \[\nu_f(A)=\sum_{i=1}^Nr_\infty^{-1}\nu_{f\circ \Psi_i}\big(\Psi_i^{-1}(A)\big).\]
\end{lemma}
\begin{proof}
    First, we assume $f|_{\partial_1 K}=0$. Then, for any $g\in \mcF_\infty$ such that $g|_{\partial_1K}=0$, we can see from Lemma \ref{lemma78} that
    \begin{equation}\label{eqn91}
        \begin{aligned}
            \int_K g(x)\nu_f(dx)&=2\mcE_\infty(fg,f)-\mcE_\infty(f^2,g)\\&=\sum_{i=1}^Nr_\infty^{-1}\Big(2\mcE_\infty\big((f\circ \Psi_i)\cdot(g\circ \Psi_i),f\circ \Psi_i\big)-\mcE_\infty(f^2\circ \Psi_i,g\circ \Psi_i)\Big)\\
            &=\sum_{i=1}^Nr_\infty^{-1}\int_{K}g\circ \Psi_i(x)\nu_{f\circ \Psi_i}(dx).
        \end{aligned}
    \end{equation}
    We can check that $\{g\in \mcF_\infty:g|_{\partial_1K}=0\}$ is dense in $\{g\in C(K):g|_{\partial_1K}=0\}$: if $g\in C(K)$ and $g|_{\partial_1K}=0$, then for any $\varepsilon>0$, there is $g'\in \mcF_\infty$ such that $\|g'-g\|_{C(K)}\leq \varepsilon$ since $(\mcE_\infty,\mcF_\infty)$ is regular, hence by letting $g''=g'-\big(g'\vee (-\varepsilon)\big)\wedge \varepsilon$, we have $g''\in \mcF_\infty,g''|_{\partial_1 K}=0$, and $\|g''-g\|_{C(K)}\leq 2\varepsilon$. So (\ref{eqn91}) holds for any $g\in C(K)$ satisfying $g|_{\partial_1K}=0$. Hence, the lemma holds for any $f\in \mcF_\infty$ with $f|_{\partial_1K}=0$.

    For general $f\in \mcF_\infty$, we consider compact $A\subset K\setminus \partial_1K$. Since $(\mcE_\infty,\mcF_\infty)$ is regular, we can find $\psi\in \mcF_\infty$ so that $\psi|_A=1$ and $\psi|_{\partial_1K}=0$. So by Lemma \ref{lemma79},
    \[\nu_f(A)=\nu_{f\psi}(A)=\sum_{i=1}^Nr_\infty^{-1}\nu_{(f\psi)\circ \Psi_i}\big(\Psi_i^{-1}(A)\big)=\sum_{i=1}^Nr_\infty^{-1}\nu_{f\circ \Psi_i}\big(\Psi_i^{-1}(A)\big).\]
    Since $\nu_f$ is a Radon measure, the lemma follows immediately.
\end{proof}

Next, we observe that the energy measure on the boundary of each cell is $0$.
\begin{lemma}\label{lemma711}
    For any $f\in \mcF_\infty$, we have $\nu_f(\partial_1K)=0$.
\end{lemma}
\begin{proof}
    By the resistance estimate in Lemma \ref{lemma77}, and by applying Theorem 15.10 and 15.11 in \cite{ki4}, we know that the Dirichlet form $(\mcE_\infty,\mcF_\infty)$ on $L^2(K,\mu)$ has a sub-Gaussian heat kernel estimate. Then, the lemma follows by applying Proposition \ref{prop24} locally on each cell $\Psi_iK$ and by \cite[Theorem 2.9]{Mathav} of Murugan.
\end{proof}

\begin{proof}[Proof of Theorem \ref{thm71}]
    (i)$\Rightarrow$(ii). If $R_n\rightarrowtail R$, then $R_n,n\geq 1$ are equicontinuous by Lemma  \ref{lemmab2} (b), hence $d_{G,n},n\geq 1$ are equicontinuous by Theorem \ref{thm52}.

    (ii)$\Rightarrow$(i). Let $n_l$, $l\geq 1$ be the subsequence, $R_\infty$ be the limit resistance metric, and $(\mcE_\infty,\mcF_\infty)$ be the limit resistance form on $K$ as in Lemma \ref{lemma77}. We first prove that $(\mcE_\infty,\mcF_\infty)$ is a self-similar $D_4$-symmetric resistance form on $K$ such that $R_\infty$ is jointly continuous and $R_\infty(L_2,L_4)=1$.

In fact, we have seen that $(\mcE_\infty,\mcF_\infty)$ and $(\mcE,\mcF)$ satisfy  the sub-Gaussian heat kernel estimates, hence by the uniqueness of the walk dimension, we know that $\mcF=\mcF_\infty$ and $r=r_\infty$. In particular, we know that
    \[
    \mcF_\infty=\{f\in l(K):\,f\circ\Psi_i\in\mcF_\infty\hbox{ for each }1\leq i\leq N\}.
    \]
    Moreover, combining Lemmas \ref{lemma710} and \ref{lemma711}, for every $f\in \mcF_\infty$, we see that
    \[
    \begin{aligned}
        \mcE_\infty(f)=\nu_{f}(K)&=\nu_{f}(K\setminus \partial_1K)\\&=\sum_{i=1}^Nr_\infty^{-1}\nu_{f\circ\Psi_i}(K\setminus\partial_0K)=\sum_{i=1}^Nr_\infty^{-1}\nu_{f\circ\Psi_i}(K)=\sum_{i=1}^Nr_\infty^{-1}\mcE_\infty(f\circ \Psi_i).
    \end{aligned}
    \]
    Hence, $(\mcE_\infty,\mcF_\infty)$ is a self-similar resistance form on $K$. The $D_4$-symmetry of $(\mcE_\infty,\mcF_\infty)$ and joinly continuity of  $R_\infty$ follow in an obvious way. It remains to prove $R_\infty(L_2,L_4)=1$.

    On one hand, we let $h_{n_l}\in\mcF_{n_l}$ such that
    \[
    h_{n_l}|_{L_2}=0,\ h_{n_l}|_{L_4}=1,\ \hbox{ and }\mcE_{n_l}(h_{n_l})=R_{n_l}(L_2,L_4)^{-1}=1 \text{ for every } l\geq 1.
    \]
    By passing to a subsequence if necessary, we can find $h\in\mcF_\infty$ such that $h_{n_l}\rightarrowtail h$, so by the $\Gamma$-convergence of $\mcE_{n_l}$ to $\mcE_\infty$,
    \[
    h|_{L_2}=0,\ h|_{L_4}=1,\ \hbox{ and }\mcE_\infty(h)\leq \liminf_{l\to\infty}\mcE_{n_l}(h_{n_l})=1.
    \]
    This implies that $R_\infty(L_2,L_4)\geq 1$.

    On the other hand, let $h_\infty\in \mcF_\infty$ such that
    \[
    h_\infty|_{L_2}=0,\ h_\infty|_{L_4}=1,\ \hbox{ and }\mcE_\infty(L_2,L_4)=R_\infty(L_2,L_4)^{-1}.
    \]
    By the $\Gamma$-convergence of $\mcE_{n_l}$ to $\mcE_\infty$, we can find $h'_{n_l}\in\mcF_{n_l}$ such that $h'_{n_l}\rightarrowtail h_\infty$ and
    \[
    1=\limsup\limits_{l\to\infty}R_{n_l}(L_2,L_4)^{-1}\leq \lim\limits_{l\to\infty}\mcE_{n_l}(h'_{n_l})=R_{\infty}(L_2,L_4)^{-1}.
    \]
    So $R_\infty(L_2,L_4)\leq 1$.

    Thus we have proved that $(\mcE_\infty,\mcF_\infty)$ is self-similar and $D_4$-symmetric, $R_\infty$ is joinly continuous and $R_\infty(L_2,L_4)=1$, so by Theorem \ref{thm61}, $R_\infty=R$. As a consequence, $R_{n_l}\rightarrowtail R$. The same result holds for any subsequence of $R_n$, so we obtain that $R_{n}\rightarrowtail R$. By Theorem \ref{thmc1}, the statement (i) follows.
\end{proof}

%\subsection*{Conflicts of interest} The Authors declare that there is no conflict of interest.

%\subsection*{Data availability statement.} Data sharing not applicable to this article as no datasets were generated or analysed during the current study.

\bibliographystyle{amsplain}

%\clearpage
%\newpage
\end{document}